\documentclass[12pt]{amsart}
 
\usepackage{amsmath}
\usepackage{amssymb}
\usepackage{caption}
\usepackage{subcaption}
\usepackage{float}
\usepackage[curve]{xypic}
\usepackage{cite}
\usepackage{enumitem}
\usepackage{color}
\usepackage{url}
\usepackage[usenames,dvipsnames]{xcolor}
\usepackage{graphicx}
\usepackage{verbatim}
\usepackage{tikz}
\usepackage[normalem]{ulem}
\usepackage{hyperref}
\hypersetup{
     colorlinks   = true,
     citecolor    = black,
     linkcolor    = blue
}

\makeatletter % makes "@" a normal character so that it can be used in macros
\def\@cite#1#2{{\m@th\upshape\bfseries%
[{#1\if@tempswa{\m@th\upshape\mdseries, #2}\fi}]}}
\makeatother % returns "@" to its normal state

\theoremstyle{plain}
\newtheorem{thm}{Theorem}[section]% subsection

\newtheorem{cor}[thm]{Corollary}
\newtheorem{prop}[thm]{Proposition}

\newtheorem{lemma}[thm]{Lemma}

\theoremstyle{definition}
\newtheorem{defn}[thm]{Definition}

\newtheorem{ass}[thm]{Assumption}

\theoremstyle{}
\newtheorem{rem}[thm]{Remark}

\numberwithin{equation}{subsection}
\newcommand{\rnc}{\renewcommand}
\newcommand{\C}{\mathbb C}
\newcommand{\Q}{\mathbb Q}

\newcommand{\R}{\mathbb R}
\newcommand{\ds}{\displaystyle}
\newcommand{\wt}{\widetilde}

\newcommand{\ra}{\longrightarrow}

\newcommand{\M}{\mathcal M}
\newcommand{\N}{\mathcal N}

\newcommand{\strata}{\mathcal H}
\newcommand{\hyp}{\mathcal{H}^{hyp}}

\DeclareMathOperator{\SL}{SL}
\DeclareMathOperator{\GL}{GL}

\newcommand{\ackn}{\noindent{\sc Acknowledgement }\hspace{5pt} }

\newcommand{\cC}{\mathcal C}
\rnc\cD{\mathcal D}

\newcommand{\cU}{\mathcal U}

\begin{document}

\title[Marked Points on Generic Surfaces]{$\GL_2 \R$-invariant Measures in Marked Strata: \\ Generic Marked Points, Earle-Kra for Strata, and Illumination}

\author[Apisa]{Paul Apisa}
\address{Department of Mathematics, University of Chicago, Chicago, IL, 60615} \email{paul.apisa@gmail.com}
  
\maketitle
    
\begin{abstract}
We classify $\GL(2, \R)$ invariant point markings over components of strata of Abelian differentials. Such point markings exist only when the component is hyperelliptic and arise from marking Weierstrass points or two points exchanged by the hyperelliptic involution. We show that these point markings can be used to determine the holomorphic sections of the universal curve restricted to orbifold covers of subvarieties of the moduli space of Riemann surfaces that contain a Teichm\"uller disk. The finite blocking problem is also solved for translation surfaces with dense $\GL(2, \R)$ orbit.
%Let $\strata$ be a component of a stratum of Abelian differentials and let $\strata(0^m, 0^n)$ be the collection of translation surfaces in $\strata$ together with $m+n$ distinct marked points, none coinciding with zeros, $m$ of which are labelled and $n$ of which are unlabelled. We show that proper affine invariant submanifolds of $\strata(0^m, 0^n)$ that push forward to $\strata$ under the forgetful map exist if and only if $\strata$ is a hyperelliptic component, in which case the affine invariant submanifolds are constructed by marking Weierstrass points and/or marking points identified under the hyperelliptic involution. In analogy to work of Earle-Kra we show that measurable equivariant sections of the forgetful map from $\strata(0^m, 0^n)$ to $\strata(0^m, \emptyset)$ exist only when $\strata$ is hyperelliptic and are constructed by marking all Weierstrass points and/or points exchanged with one of the $m$ labelled points under the hyperelliptic involution. As a corollary the illumination and finite blocking problems are resolved for generic translation surfaces. Combined with previous work of the author, these results provide a classification of all higher rank orbit closures in $\strata(0^m, 0^n)$ when $\strata$ is a hyperelliptic component.
\end{abstract}

%%%%%%%%%%%%%%
%
% SECTION - INTRODUCTION
%
%%%%%%%%%%%%%%%

\section{Introduction}

In this paper, we consider the $\GL(2, \R)$ action on strata of translation surfaces with marked points and give applications to the study of holomorphic sections of the universal curve and the finite blocking problem. \\ \vspace{-2mm}

\noindent \textbf{Background.} Let $\mathcal{Q} \M_g$ be the moduli space of quadratic differentials on genus $g$ Riemann surfaces. The space admits a $\GL(2, \R)$ action arising from Teichm\"uller geodesic flow and scalar multiplication. The space $\mathcal{Q} \M_g$ also admits a $\GL(2, \R)$-invariant stratification given by specifying the number of zeros and poles of the quadratic differentials and their orders.
 
In the sequel we will use the notation $(X, q)$ for a point of $\mathcal{Q} \M_g$ where $X$ is a genus $g$ Riemann surface and $q$ is a quadratic differential. The point $(X, q)$ will be called generic if its $\GL(2, \R)$ orbit is dense in the stratum containing it. Similarly, a collection of points $P$ on $(X, q)$ will be said to be generic if the complex dimension of the $\GL(2, \R)$ orbit closure of $(X, q; P)$ is $|P|$ more than the dimension of the $\GL(2, \R)$ orbit closure of $(X, q)$. A point $p$ on $(X, q)$ that does not coincide with a zero or pole of $q$ and so that $\{ p \}$ is not generic will be called a periodic point. We will assume throughout that $P$ contains no zeros or poles of $q$. \\ \vspace{-2mm}
 
\noindent \textbf{Holomorphic Sections of the Universal Curve.} In Hubbard~\cite{Hub}, it was shown that holomorphic sections of the universal curve over $\mathrm{Teich}_{g,n}$ - the Teichm\"uller space of genus $g$ Riemann surfaces with $n$ punctures - exist only when $(g, n) = (2,0)$, in which case the sections mark fixed points of the hyperelliptic involution. Earle and Kra~\cite{EK} generalized this result by allowing the sections to mark punctured points. They showed that in this more general setting, the only new sections that could arise were in genus two and are given by taking a punctured point and marking its image under the hyperelliptic involution. 

Inspired by these questions, we show the following. Let $C$ be a subvariety of $\M_{g, n}$ and let $\Gamma$ be a torsionfree finite index subgroup of the mapping class group $\mathrm{Mod}_{g,n}$. Let $C(\Gamma)$ be the preimage of $C$ on $\mathrm{Teich}_{g,n} / \Gamma$.

\begin{thm}\label{T2}
If $C$ contains the Teichm\"uller disk generated by the quadratic differential $(X, q)$, then any holomorphic section of the universal curve over $C(\Gamma)$ marks either a pole, zero, or periodic point of $(X,q)$ on $X$.
\end{thm}

\noindent This result holds for any section defined over a component of $C(\Gamma)$ as well.  By Eskin-Filip-Wright~\cite[Theorem 1.5]{EFW}, there are finitely many periodic points on $(X, q)$ if and only if its holonomy double cover is not a torus cover. For closures $C$ of Teichm\"uller disks generated by such quadratic differentials, there are only finitely many points that can be marked by holomorphic sections of the universal curve defined over $C(\Gamma)$ for any torsionfree finite index subgroup $\Gamma$ of the mapping class group. 

\begin{rem}\label{R:list}
Apart from $\M_{g,n}$ itself, other examples of algebro-geometrically interesting varieties that contain a Teichm\"uller disk are the theta-null divisor (see M\"uller~\cite{Muller-theta} and Grushevsky-Zakharov~\cite{Grushevsky-Zakharov}), the anti-ramification locus (see Farkas-Verra~\cite{Farkas-Verra-theta}), and the Weierstrass divisor (see Cukierman~\cite{Cukierman}). More examples are listed in Mullane~\cite{Mullane1} and the Kodaira dimensions of many such loci are computed in Gendron~\cite{Gendron-Incidence}. 
\end{rem}

\begin{rem}
McMullen classified all $\GL(2, \R)$ orbit closures in $\Omega \M_2$ in~\cite{Mc},~\cite{McM:spin},~\cite{Mc4}, and~\cite{Mc5}. A corollary of his work is that loci of surfaces in $\M_2$ whose Jacobians admit real multiplication by a fixed order in a real quadratic field contain Teichm\"uller disks. In Apisa~\cite{Apisa-golden}, the author will apply Theorem~\ref{T2} to show that sections of the universal curve defined over orbifold covers of these loci (as in Theorem~\ref{T2}) exist if and only if the locus is the locus of genus two Riemann surfaces whose Jacobians admit real multiplication by the maximal order in $\Q[\sqrt{5}]$.
\end{rem}

\noindent We will presently see that holomorphic sections of the universal curve over the loci in Remark~\ref{R:list} may be classified with the methods of this paper. \\ \vspace{-2mm}

\noindent \textbf{Finite Blocking on Translation Surfaces.} The finite blocking problem asks whether given two not necessarily distinct points $p$ and $q$ on a translation surface, there is a finite collection of points that intersects every straight line path from $p$ to $q$. We will continue to assume that neither $p$ nor $q$ coincide with a singularity of the flat metric.

\begin{thm}\label{T3}
A generic translation surface contains a pair of finitely blocked points if and only if it belongs to a hyperelliptic component, in which case the pair consists of a point and its image under the hyperelliptic involution.
\end{thm} 

\noindent We will see shortly that Theorem~\ref{T3} follows from the classification of non-generic points on generic translation surfaces (Theorem~\ref{T1}) and the methods of Apisa-Wright~\cite{Apisa-Wright}. For more on the finite blocking problem see Leli\`evre-Monteil-Weiss~\cite{LMW}. For applications to billiards see Apisa-Wright~\cite[Section 3]{Apisa-Wright}.  \\ \vspace{-2mm}

\noindent \textbf{Marked Points on Generic Translation Surfaces.} 

As in the case of Hubbard~\cite{Hub} and Earle-Kra~\cite{EK}, the following result shows that the only non-generic collection of marked points on a generic translation surface arise from hyperellipticity.

\begin{thm}\label{T1}
If $(X, \omega)$ is a generic translation surface of genus at least two and $P$ is a non-generic collection of points then $(X, \omega)$ belongs to a hyperelliptic component of a stratum of Abelian differentials and $P$ contains either a Weierstrass point or two points exchanged by the hyperelliptic involution.
\end{thm}

\begin{rem}
Each subvariety $C$ of $\M_{g,n}$ mentioned in Remark~\ref{R:list} is the projection of a non-hyperelliptic component $\strata$ of some stratum of Abelian differentials to $\M_{g,n}$. It follows from Theorem~\ref{T2} and Theorem~\ref{T1} that if $(X, \omega)$ belongs to $\strata$ then any holomorphic sections of the universal curve defined over $C$ can only mark zeros of $\omega$ over $X$.
\end{rem}

\noindent Finally, we remark that Theorem~\ref{T1} immediately implies Theorem~\ref{T3}. A key observation is that the collection of translation surfaces represented by strictly convex $2n$-gons with opposite sides identified is open and $\GL(2, \R)$-invariant in hyperelliptic components of strata. By Eskin-Mirzakhani-Mohammadi~\cite{EMM}, this implies that the complement of this collection of surfaces is a finite union of proper affine invariant submanifolds and hence only consists of non-generic translation surfaces. 

\begin{proof}[Proof of Theorem~\ref{T3} given Theorem~\ref{T1}:]
By Apisa-Wright~\cite[Theorem 3.15]{Apisa-Wright} and Theorem~\ref{T1}, finitely blocked pairs of points only occur when the generic translation surfaces belongs to a hyperelliptic component; and in this case, the pair of points consists of either two Weierstrass points or two points exchanged by the hyperelliptic involution and in both cases the finite blocking set is the collection of Weierstrass points. Since every generic translation surface in a hyperelliptic component may be represented by a strictly convex $2n$-gon with opposite sides identified, with Weierstrass points being the midpoints of the polygon and its edges (and the vertices when $n$ is even), we see from convexity that a Weierstrass point is at most finitely blocked from itself. Therefore, the collection of points finitely blocked from each other are exactly the ones containing a point and its image under the hyperelliptic involution.  
\end{proof} %\text{} %\\ \vspace{-2mm}

\noindent \textbf{Organization.} The proof of Theorem~\ref{T2} is independent of the rest of the paper and appears in Section~\ref{S:T2}. The outline of the proof of Theorem~\ref{T1} is given in Section~\ref{S:T1} and reduced to two more technical results that are established in the subsequent two sections. The two main tools used in the proof of Theorem~\ref{T1} are the construction of generic horizontally and vertically periodic translation surfaces in every component of every stratum of holomorphic one-forms in Section~\ref{S:KZ-Surfaces} and results in Section~\ref{S:Cylinder} that constrain the positions of periodic points using cylinders.
%The proof of this theorem, which appears in Section~\ref{S:T1}, will involve first constructing generic horizontally and vertically periodic translation surfaces in each component of each stratum of Abelian differentials (Section~\ref{S:KZ-Surfaces}) and then using cylinder deformations to study marked points (Section~\ref{S:Cylinder}).

\bigskip

\ackn The author thanks Alex~Eskin for suggesting the problem and thanks Alex~Eskin, Alex~Wright, and Curt~McMullen for their insightful comments. He thanks Ronen~Mukamel for suggesting the parallel between the main theorems and the work of Hubbard and Earle-Kra. He thanks Matt Bainbridge for suggesting the connection to the problem of classifying holomorphic sections. This material is based upon work supported by the National Science Foundation Graduate Research Fellowship Program under Grant No. DGE-1144082. The author gratefully acknowledges their support.

%%%%%%%%%%%%%%%%%%%%%%%%%%%%%%%%%
%
%
% Background
%
%
%%%%%%%%%%%%%%%%%%%%%%%%%%%%%%%%%

\section{Background}

In this section, we will summarize the main tools used in the sequel. Let $\Omega \M_g$ be the moduli space of holomorphic one-forms (equivalently Abelian differentials or translation surfaces). As mentioned above this space admits a $\GL(2, \R)$ action and a $\GL(2, \R)$-invariant stratification. Each stratum of Abelian differentials admits a coordinate system - period coordinates - given by specifying the periods of the holomorphic one-forms in a basis of homology relative to the collection of zeros. The change of coordinates between two charts in period coordinates is a constant volume-preserving linear function, which endows strata of $\Omega \M_g$ with a linear structure and with a well-defined Lebesgue measure, see Zorich~\cite{Z} for details.

Lebesgue measure on strata of $\Omega \M_g$ induces a finite $\SL(2, \R)$-invariant measure on $\cU$ - the locus of unit area translation surfaces. An affine invariant submanifold is a closed subset of a stratum of $\Omega \M_g$ that is linear in period coordinates. Eskin and Mirzakhani~\cite{EM} showed that the collection of $\SL(2, \R)$-invariant ergodic measures on $\cU$ are precisely Lebesgue measure on affine invariant submanifolds restricted to $\cU$. Eskin-Mirzakhani-Mohammadi~\cite{EMM} showed that orbit closure of any holomorphic one-form in a stratum is exactly an affine invariant submanifold. 

\begin{defn}
Given an affine invariant submanifold $\M$, let $\mu_\M$ be Lebesgue measure on $\M \cap \cU$. Given a $\GL(2, \R)$-equivariant measurable map $f: \M \ra \N$ between two affine invariant submanifolds, define the pushforward of $\M$, denoted $f_* \M$, to be the affine invariant submanifold represented by the ergodic $\SL(2, \R)$-invariant measure given by $f_* \mu_\M$. Notice that $f(\M)$ coincides with $f_* \M$ up to sets of measure zero. 
\end{defn}

A fundamental tool in the sequel will be cylinder deformations. 

\begin{defn}
Suppose that $(X, \omega)$ is a translation surface in an affine invariant submanifold $\M$. If $C_1$ and $C_2$ are two cylinders in $\M$ then $C_1$ and $C_2$ will be said to be $\M$-equivalent if $C_1$ and $C_2$ are parallel at all surfaces in an open neighborhood of $(X, \omega)$ in $\M$. A maximal collection of equivalent cylinders on $(X, \omega)$ will be called an $\M$-equivalence class. Given a collection $\cC$ of cylinders, the standard shear, $\sigma_\cC$, and the standard dilation, $a_\cC$, are the following cohomology classes
\[ \sigma_\cC := \sum_{c \in \cC} h_c \gamma_c^* \qquad a_\cC := i \sum_{c \in \cC} h_c \gamma_c^* \]
where $h_c$ is the height of cylinder $c$ and $\gamma_c$ is the core curve of the cylinder $c$ oriented from left to right.
\end{defn}

\begin{thm}[Wright~\cite{Wcyl}, Corollary 3.4]\label{T:W1}
Suppose that $(X, \omega)$ is a horizontally periodic translation surface whose $\GL(2, \R)$ orbit closure is $\M$. Let $(C_1, \hdots, C_n)$ be an enumeration of the horizontal cylinders and suppose that cylinder $C_i$ has modulus $m_i$ and core curve (oriented from left to right) $\gamma_i$ of length $c_i$ for $i = 1, \hdots, n$. Let $W \subseteq \Q^n$ be the subset of rational homogeneous linear relations that the moduli $(m_i)_{i=1}^n$ satisfy, i.e. $w \in W$ if and only if $w \cdot m = 0$. If $(v_i)_{i=1}^n \in \C^n$ belongs to $W^\perp$, then 
\[ \sum_{i=1}^n c_i v_i \gamma_i^* \in T_{(X, \omega)} \M \]
\end{thm}

\begin{thm}[Wright~\cite{Wcyl}, Lemma 4.11]\label{T:W2}
If $\cC$ is an equivalence class of horizontal cylinders on a translation surface $(X, \omega)$ contained in affine invariant submanifold $\M$ then the standard shear and the standard dilation are both contained in $T_{(X, \omega)} \M$.
\end{thm}

Finally, the fundamental object of study in the sequel will be the following.

\begin{defn}
Given an affine invariant submanifold $\M$, let $\M(0^n)$ be the collection of quadratic differentials $(X, q) \in \M$ together with $n$ distinct marked points that do not coincide with zeros or poles of $q$ (where $n$ is a positive integer). Let $\overline{\M(0^n)}$ be the partial compactification of $\M(0^n)$ where marked points are allowed to coincide with each other and with zeros and poles of $q$. Both $\M(0^n)$ and $\overline{\M(0^n)}$ admit a $\GL(2, \R)$-action and all results stated in this section continue apply to both of these spaces.
\end{defn}

Notice that $\overline{\M(0^n)}$ is contained in the Mirzakhani-Wright partial compactification of $\M(0^n)$ (see~\cite{MirWri})

\begin{thm}[Mirzakhani-Wright~\cite{MirWri}, Corollary 1.2]
Suppose that $\M$ is an affine invariant submanifold and $(X_n, \omega_n)$ is a sequence of points in $\M$ that converges to a, possibly disconnected, translation surface $(Y, \eta)$ in the boundary of $\M$. The orbit closure of any component of $(Y, \eta)$ has smaller dimension than $\dim \M$. 
\end{thm}

\begin{lemma}\label{L:fiber}
Let $\M$ be an affine invariant submanifold and let $\N$ be an affine invariant submanifold in $\M(0)$. Let $\pi: \M(0) \ra \M$ be the map that forgets the marked point. If $\pi_* \N = \M$ then $\pi\left( \N \right)$ is an open dense subset of $\M$ and the fiber over a generic point in $\M$ is nonempty. Moreover, $\pi$ is open. 
\end{lemma}
\begin{proof}
Since all fibers of the forgetful map $\pi: \overline{\M(0)} \ra \M$ are compact, $\pi$ is proper and its image is closed. The image of $\pi$ is full measure in $\pi_* \N$, which is $\M$ by assumption. If $\N = \M(0)$ then the result is immediate, so suppose instead that $\N$ and $\M$ have the same dimension. By Mirzakhani-Wright~\cite[Corollary 1.2]{MirWri}, the components of the boundary of $\overline{\N} \cap \overline{\M(0)}$ have strictly smaller dimension than $\M$ and hence their pushfoward $\cC$ cannot be $\M$ by Sard's theorem. Therefore, $\pi\left( \N \right)$ contains the complement of $\cC$, which is an open dense set in $\M$. Since $\pi: \N \ra \M$ is a finite holomorphic map between equidimensional varieties, it is open.

%To see finiteness suppose to a contradiction that $p \in \M$ has fiber $C$ in $\M(0)$ and that $C \cap \N$ is not finite. Since $C \cap \N$ is an algebraic subvariety of $C$ and since it has an accumulation point, it coincides with the entire fiber. Therefore, $\N$ must contain a tangent vector that corresponds to varying $p$ freely. This implies that $\N = \M(0)$ a contradiction. 
\end{proof}

\begin{defn}
If $\M$ is an affine invariant submanifold then let $\M^{ord}(0^n)$ be the finite cover of $\M(0^n)$ where the marked points are labelled. Let $\pi_k: \M(0^n) \ra \M(0^{n-1})$ be the map that forgets the $k$th-marked point where $k \in \{1, \hdots, n\}$. 
\end{defn}

\begin{lemma}\label{L:neighbor}
Let $\N$ be an affine invariant submanifold in $\M(0^n)$ that pushes forward to $\M$ under the map $\pi$ that forgets all marked points. If $(X, \omega; P)$ is generic in $\M(0^n)$, then there is an open set $U$ of $(X, \omega; P)$ in $\N$ so that $\pi(U)$ is an open set around $(X, \omega)$ in $\M$.
\end{lemma}
\begin{proof}
Without loss of generality, we will work on $\M^{ord}(0^n)$. Since each $\pi_k$ restricted to $\N$ is open by Lemma~\ref{L:fiber} and since $\pi = \pi_1 \circ \hdots \circ \pi_n$, $\pi$ restricted to $\N$ is open as well.  
\end{proof}

%%%%%%%%%%%%%%%%%%%%%%%%%%%%%%%%%
%
%
% SECTION - Proof of Holomorphic Sections
%
%
%%%%%%%%%%%%%%%%%%%%%%%%%%%%%%%%%

\section{Holomorphic Sections over Varieties containing a Teichm\"uller Disk - Proof of Theorem~\ref{T2}}\label{S:T2}

Throughout this section we will make the following assumption.

\begin{ass}\label{A:sections}
Let $\Gamma$ be a torsionfree finite index subgroup of the mapping class group. Let $C$ be a complex analytic subvariety of $\mathrm{Teich}_{g,n} / \Gamma$ where $\mathrm{Teich}_{g,n}$ is the Teichm\"uller space of a closed genus $g$ surface with $n$ punctures and so that $3g-3+n > 0$.  Let $\pi: \mathcal{C}_{g,n} \ra \mathrm{Teich}_{g,n} / \Gamma$ be the universal curve. Suppose additionally that $C$ contains a Teichm\"uller disk generated by the quadratic differential $(X, q)$ where $X$ is a Riemann surface and $q$ a quadratic differential on $X$. Suppose that $\mathcal{Q}$ is the stratum of quadratic differentials to which $(X, q)$ belongs and let $\M$ be the orbit closure of $(X, q)$ in $\mathcal{Q}$. %Notice that $\M$ is contained in the cotangent bundle of $C$. Let $\overline{\M(0) }$ be the partial compactification of $\M(0)$ where the marked points are allowed to collide with zeros and poles of the underlying quadratic differential. This is a closed $\GL(2, \R)$-invariant subset of the Mirzakhani-Wright partial compactification of $\M(0)$ (see~\cite{MirWri}).
\end{ass}

%\begin{ass}\label{A:sections}
%Let $\Gamma$ be a torsionfree finite index subgroup of the mapping class group. Let $C$ be a complex analytic subvariety of $\mathrm{Teich}_{g,n} / \Gamma$ where $\mathrm{Teich}_{g,n}$ is the Teichm\"uller space of a closed genus $g$ surface with $n$ punctures and so that $3g-3+n > 0$. Suppose additionally that $C$ contains a Teichm\"uller disk generated by the quadratic differential $(X, q)$ where $X$ is a Riemann surface and $q$ a quadratic differential on $X$. Suppose that $\mathcal{Q}$ is the stratum of quadratic differentials to which $(X, q)$ belongs and let $\M$ be the orbit closure of $(X, q)$. Notice that $\M$ is contained in the cotangent bundle of $C$. Let $\overline{\M(0) }$ be the partial compactification of $\M(0)$ where the marked points are allowed to collide with zeros and poles of the underlying quadratic differential. This is a closed $\GL(2, \R)$-invariant subset of the Mirzakhani-Wright partial compactification of $\M(0)$ (see~\cite{MirWri}). Let $\pi: \mathcal{C}_{g,n} \ra \mathrm{Teich}_{g,n} / \Gamma$ be the universal curve. 
%\end{ass}

\begin{lemma}\label{L:equivariant}
Every holomorphic section of $\pi$ defined over $C$ induces a $\GL(2, \R)$-equivariant section of the forgetful map from $\overline{\M(0) }$ to $\M$.
\end{lemma}
\begin{proof}
Let $s: C \ra \cC_{g,n}$ be a holomorphic section of $\pi$. Let $\iota: \mathbb{D} \ra C$ be the inclusion of the Teichm\"uller disk  into $C$. The inclusion is an isometry with respect to the underlying Kobayashi hyperbolic metrics.  Since $\iota = \pi \circ s$ and $\pi$ and $s$ are contractions in the Kobayashi hyperbolic metrics, it follows that $s$ restricted to the embedded Teichm\"uller disk in $C$ is an isometry in the Kobayashi metrics. Therefore, $s\left( \iota\left( \mathbb{D} \right) \right)$ is a Teichm\"uller disk in $\cC_{g,n}$. 

Sufficiently close Riemann surfaces $X_1$ and $X_2$ contained in $\iota \left( \mathbb{D} \right)$ are joined by a dilatation minimizing homeomorphism given by a geodesic in $\iota\left( \mathbb{D} \right)$. By Teichm\"uller's theorem, this homeomorphism is unique up to pre- and post-composition with a conformal automorphism. However, since $\Gamma$ is torsionfree there are no such automorphisms that fix $X_1$ or $X_2$. If $\gamma$ is the geodesic from $X_1$ to $X_2$ in $\iota\left( \mathbb{D} \right)$, then $s(\gamma)$ is a geodesic of the same length in $\mathcal{C}_{g,n}$ and hence corresponds to a homeomorphism with the same dilatation. By uniqueness this path must correspond to the same homeomorphism and so if Teichm\"uller geodesic flow along $(X_1, q_1)$ produces the geodesic $\gamma$, Teichm\"uller geodesic flow along $(s(X), q)$ produces $s(\gamma)$. 

Let $\wt{s}: \M \ra \overline{\M(0) }$ be the section of the forgetful map from $\overline{\M(0) }$ to $\M$ given by sending a quadratic differential $(X, q)$ to $\left( s(X), q \right)$. The argument above shows that this map is $\GL(2, \R)$-equivariant on Teichm\"uller disks since it is equivariant under complex scalar multiplication and Teichm\"uller geodesic flow. Since $\M$ is foliated by $\GL(2, \R)$ invariant Teichm\"uller disks the claim follows.
\end{proof}

\begin{rem}
The proof of Lemma~\ref{L:equivariant} only uses the hypothesis that $\Gamma$ is torsionfree, not that it is finite index.
\end{rem}
%
%\begin{lemma}
%Every holomorphic section of $\pi$ defined over $C$ marks a periodic points, zero, or pole of $(X, q)$ over $X$.
%\end{lemma}
\begin{proof}[Proof of Theorem~\ref{T2}:]
Let $s$ be a holomorphic section of $\pi$ defined over $C$. Since it is a continuous section, it follows that $s(C)$ and hence $\wt{s}\left( \M \right)$ is closed. By Lemma~\ref{L:equivariant}, $\wt{s}\left( \M \right)$ is closed and $\GL(2, \R)$-invariant and therefore it is an affine-invariant submanifold by Eskin-Mirzakhani-Mohammadi~\cite{EMM}. Notice that this application of Eskin-Mirzakhani-Mohammadi uses the fact that $\Gamma$ is finite-index since this implies that the Lebesgue measure of the collection of unit area half-translation surfaces in $\M$ is finite. Since $\wt{s}\left( \M \right)$ is an affine invariant-submanifold that does not coincide with $\M(0)$, it follows that the point that $s$ marks above $X$ is a periodic point, zero, or pole of $(X, q)$.
\end{proof}

\begin{rem}
The same proof shows that measurable equivariant sections of the forgetful map from $\overline{\M}(0)$ to $\M$ only mark periodic points, zeros, or poles. In the measurable setting, the section is used to pushforward Lebesgue measure to a measure on  $\overline{\M(0)}$, which must be Lebesgue measure on an affine invariant submanifold by Eskin-Mirzakhani~\cite{EM}. The details are omitted.
%In that proof, since continuity cannot be used, Lebesgue measure on the unit area translation surfaces in $\M$ - which is an ergodic probability measure (after scaling) - is pushed forward to an ergodic probability measure on $\overline{\M}(0)$, which is an ergodic $\GL(2, \R)$-invariant probability measure, and hence one supported on a finite union of affine invariant submanifolds by Eskin-Mirzakhani~\cite{EM}. Notice that Eskin-Mirzakhani~\cite{EM} and not Eskin-Mirzakhani-Mohammadi~\cite{EMM} is used precisely because the section is only assumed to be measurable. We omit the details since the only such measurable and equivariant sections known to the author are holomorphic.
\end{rem}

%%%%%%%%%%%%%%%%%%%%%%%%%%%%%%%%%
%
%
% SECTION - GENERIC TRANSLATION SURFACES IN EVERY COMPONENT
%
%
%%%%%%%%%%%%%%%%%%%%%%%%%%%%%%%%%

\section{Explicit Translation Surfaces in Every Component of Every Stratum}\label{S:KZ-Surfaces}

In this section we will construct explicit generic translation surfaces in each connected component of every stratum of Abelian differentials. The connected components were classified by Kontsevich and Zorich~\cite{KZ}.

\begin{thm}[Kontsevich-Zorich~\cite{KZ}] All strata are connected except for the following:
\begin{itemize}
\item For $g>3$, $\mathcal{H}(2g-2)$ has three connected components characterized by odd spin, even spin, and hyperellipticity.
\item For odd $g > 3$, $\mathcal{H}(g-1, g-1)$ has three connected components characterized by odd spin, even spin, and hyperellipticity.
\item For even $g>3$, $\mathcal{H}(g-1,g-1)$ has two connected components characterized by hyperellipticity and nonhyperellipticity. 
\item For $g>3$, $\mathcal{H}(2k_1, \hdots, 2k_n)$ has two connected components characterized by odd and even spin (excluding the case $\mathcal{H}(g-1,g-1)$ for odd $g>3$, which, as mentioned above, has three components).
\item $\strata(4)$ and $\strata(2,2)$ have two connected components - a hyperelliptic and an odd one.
\end{itemize}
\end{thm} 

To distinguish which connected component of a stratum a specific translation surface belongs to, we will use the following criterion.

\begin{thm}[Kontsevich-Zorich~\cite{KZ}, Corollary 2]
Let $\strata$ be a stratum of Abelian differentials. For each connected component $C$ of the minimal stratum there is a unique component of $\strata$ that contains $C$ in its closure.
\end{thm}

We are now in a position to create horizontally and vertically periodic translation surfaces in each component of each stratum. First, we establish a convention:

\textbf{Convention for Figures:} We will often use polygons, all of whose edges will be vertical or horizontal, to represent translation surfaces using the following two conventions. The edge of a polygon will mean a line segment in the boundary of the polygon that connects two vertices and has no vertex in its interior.
\begin{enumerate}
\item The intersection of a dotted line and an edge is a vertex of the polygon.
\item If a pair of unmarked vertical (resp. horizontal) edges contain interior points that can be connected by a horizontal (resp. vertical) line that lies in the interior of the polygon then they are identified.
\end{enumerate}
Under this convention the two translation surfaces in Figure~\ref{fig:important-convention} are identical:

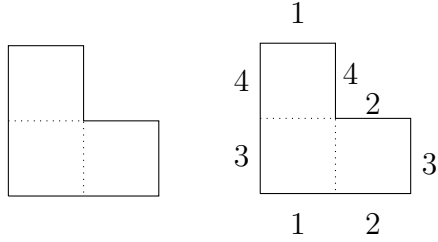
\begin{figure}[H]
	\begin{subfigure}[b]{0.2\textwidth}
        \centering
        		\begin{tikzpicture}
			\draw (0,1) -- (0,3) -- (1,3) -- (1,2) -- (2,2) -- (2,1) -- (0,1);
			\draw[dotted] (0,2) -- (1,2) -- (1,1);	
			\node at (-.5, .5) {$$};
		\end{tikzpicture}
	\end{subfigure}
	\qquad
        \begin{subfigure}[b]{0.2\textwidth}
        \centering
      	\begin{tikzpicture}
		\draw (0,0) -- (0,2) -- (1,2) -- (1,1) -- (2,1) -- (2,0) -- (0,0);
		\draw[dotted] (0,1) -- (1,1) -- (1,0);	
		\node at (.5,2.4) {$1$}; \node at (.5,-0.4) {$1$};
		\node at (1.5,1.2) {$2$}; \node at (1.5,-0.4) {$2$};
		\node at (-.25,.5) {$3$}; \node at (2.25,0.4) {$3$};
		\node at (-.25,1.5) {$4$}; \node at (1.2,1.6) {$4$};	
	\end{tikzpicture}
         \end{subfigure}
\caption{Equivalent representations of the same translation surface} 
\label{fig:important-convention}
\end{figure}
Using separatrix diagrams, Kontsevich and Zorich produce surfaces that belong to each component of the minimal stratum, see~\cite[Figure 4]{KZ}. The surfaces are represented as translation surfaces in Figure~\ref{fig:KZ-min}.

\begin{figure}[H]
    \begin{subfigure}[b]{0.3\textwidth}
        \centering
        \resizebox{\linewidth}{!}{\begin{tikzpicture}
		\draw (0,0) -- (0,1) -- (1,1) -- (1,2) -- (1.95,2) -- (1.95,1) -- (2.05,1) -- (2.05,2) -- (2.95,2) -- (2.95,1) -- (3.05, 1) -- (3.05,2) -- (4,2) -- (4,1);
		\draw[black, fill] (4.25,1.5) circle[radius=1pt];
		\draw[black, fill] (4.5,1.5) circle[radius=1pt];
		\draw[black, fill] (4.75,1.5) circle[radius=1pt];
		\draw (5,1) -- (5,2) -- (6,2) -- (6,1) -- (7,1) -- (7,0) -- (0,0);
		\draw[dotted] (1,0) -- (1,1) -- (2,1) -- (2,0) -- (2,1) -- (3,1) -- (3,0) -- (3,1) -- (4,1) -- (4,0);
		\draw[dotted] (5,0) -- (5,1) -- (6,1) -- (6,0);
		\node at (1.5,2.4) {$1$}; \node at (1.5,-0.4) {$1$};
		\node at (2.5,2.4) {$2$}; \node at (2.5,-0.4) {$2$};
		\node at (3.5,2.4) {$3$}; \node at (3.5,-0.4) {$3$};
		\node at (5.5,2.4) {$g-1$}; \node at (5.5,-0.4) {$g-1$};
	\end{tikzpicture}
        }
        \caption{$\strata^{odd}(2g-2)$}
        \label{fig:subfig1}
    \end{subfigure}
        \begin{subfigure}[b]{0.3\textwidth}
        \centering
        \resizebox{\linewidth}{!}{\begin{tikzpicture}
		\draw (0,0) -- (0,1) -- (1,1) -- (1,2) -- (1.95,2) -- (1.95,1) -- (2.05,1) -- (2.05,2) -- (2.95,2) -- (2.95,1) -- (3.05, 1) -- (3.05,2) -- (4,2) -- (4,1);
		\draw[black, fill] (4.25,1.5) circle[radius=1pt];
		\draw[black, fill] (4.5,1.5) circle[radius=1pt];
		\draw[black, fill] (4.75,1.5) circle[radius=1pt];
		\draw (5,1) -- (5,2) -- (6,2) -- (6,1) -- (7,1) -- (7,0) -- (0,0);
		\draw[dotted] (1,0) -- (1,1) -- (2,1) -- (2,0) -- (2,1) -- (3,1) -- (3,0) -- (3,1) -- (4,1) -- (4,0);
		\draw[dotted] (5,0) -- (5,1) -- (6,1) -- (6,0);
		\node at (1.5,2.4) {$1$}; \node at (1.5,-0.4) {$2$};
		\node at (2.5,2.4) {$2$}; \node at (2.5,-0.4) {$1$};
		\node at (3.5,2.4) {$3$}; \node at (3.5,-0.4) {$3$};
		\node at (5.5,2.4) {$g-1$}; \node at (5.5,-0.4) {$g-1$};
	\end{tikzpicture}
        }
        \caption{$\strata^{even}(2g-2)$}
        \label{fig:KZ-minB}
    \end{subfigure}
    \begin{subfigure}[b]{0.3\textwidth}
        \centering
        \resizebox{\linewidth}{!}{\begin{tikzpicture}
		\draw (0,0) -- (0,1) -- (1,1) -- (1,2) -- (1.95,2) -- (1.95,1) -- (2.05,1) -- (2.05,2) -- (2.95,2) -- (2.95,1) -- (3.05, 1) -- (3.05,2) -- (4,2) -- (4,1);
		\draw[black, fill] (4.25,1.5) circle[radius=1pt];
		\draw[black, fill] (4.5,1.5) circle[radius=1pt];
		\draw[black, fill] (4.75,1.5) circle[radius=1pt];
		\draw (5,1) -- (5,2) -- (6,2) -- (6,1) -- (7,1) -- (7,0) -- (0,0);
		\draw[dotted] (1,0) -- (1,1) -- (2,1) -- (2,0) -- (2,1) -- (3,1) -- (3,0) -- (3,1) -- (4,1) -- (4,0);
		\draw[dotted] (5,0) -- (5,1) -- (6,1) -- (6,0);
		\node at (1.5,2.4) {$1$}; \node at (1.5,-0.4) {$g-1$};
		\node at (2.5,2.4) {$2$}; \node at (2.5,-0.4) {$g-2$};
		\node at (3.5,2.4) {$3$}; \node at (3.5,-0.4) {$g-3$};
		\node at (5.5,2.4) {$g-1$}; \node at (5.5,-0.4) {$1$};
	\end{tikzpicture}
        }
        \caption{$\strata^{hyp}(2g-2)$}
        \label{fig:KZ-minC}
    \end{subfigure}
\caption{Surfaces in each component of the minimal stratum} 
\label{fig:KZ-min}
\end{figure}
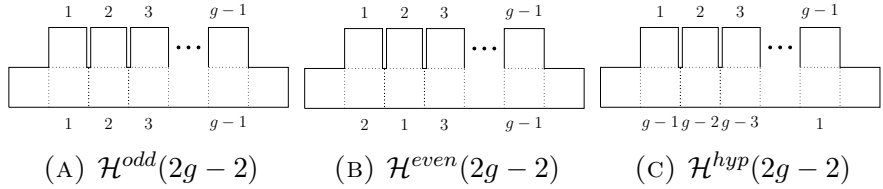

\begin{prop}[Genericity Criterion]\label{P:GC}
Suppose that $(X, \omega)$ is a translation surface in a component $\strata$ of a stratum of Abelian differentials of genus $g$ surfaces with $n$ zeros and no marked points. If $(X, \omega)$ has $g+n-1$ horizontal cylinders whose moduli satisfy no rational linear relation, then $(X, \omega)$ is generic.
\end{prop}
\begin{proof}
Let $\M$ be the orbit closure of $(X, \omega)$. It suffices to show that $\M$ coincides with a component of $\strata$. By Wright~\cite[Corollary 3.4]{Wcyl} (Theorem~\ref{T:W1}), since the moduli satisfy no rational linear relation the tangent space of $\M$ at $(X, \omega)$ includes $\{ \gamma_1^*, \hdots, \gamma_{g+n-1}^*\}$ where $\gamma_i$ are core curves oriented from left to right of the horizontal cylinders and $\gamma_i^*$ denotes the dual cohomology class under the intersection pairing. For details on the identification of $T_{(X, \omega)} \M$ with a subspace of $H^1(X, \Sigma; \C)$, for $\Sigma$ the zero set of $\omega$, see Avila, Eskin, M\"oller~\cite{AEM}. The dual cohomology classes span a complex vector space of dimension $g+n-1$.

Let $p: T_{(X, \omega)} \M \ra H^1(X, \C)$ be the projection from relative to absolute cohomology. By Avila, Eskin, M\"oller~\cite{AEM} the image of the projection is a complex symplectic vector space. The kernel of the projection has (complex) dimension at most $n-1$. Since the projection of $\{\gamma_1^*, \hdots, \gamma_{g+n-1}^*\}$ spans an isotropic subspace, which has dimension at most $g$, it follows that the kernel of $p$ has dimension exactly $n-1$ and that the projection of $\{\gamma_1^*, \hdots, \gamma_{g+n-1}^*\}$ spans a Lagrangian subspace. Since $p\left( T_{(X, \omega)} \M \right)$ is complex symplectic it follows that $p$ is a surjection with maximal dimensional kernel. It follows that $T_{(X, \omega)} \M$ is isomorphic to $H^1(X, \Sigma; \C)$ and hence that $\M$ has full dimension. Since $\M$ is open and closed it must coincide with a component of $\strata$. 
\end{proof}

\subsection*{Generic Surfaces in $\strata^{hyp}(2g-2)$ and $\strata^{hyp}(g-1,g-1)$ } \text{}

It is straightforward to verify that the translation surfaces in Figure~\ref{fig:hyp-model} are in the indicated components; see for example~\cite[Section 2]{Apisa-hyp}. The genericity criterion (Proposition~\ref{P:GC}) implies that the translation surfaces are generic provided that all moduli of horizontal cylinders satisfy no rational linear relation.

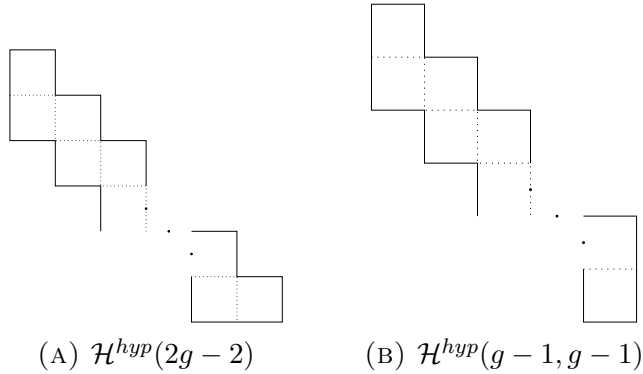
\begin{figure}[H]
    \begin{subfigure}[b]{0.3\textwidth}
        \centering
        \resizebox{\linewidth}{!}{
\begin{tikzpicture}
\draw (2, -4) -- (2, -3) -- (1, -3) -- (1, -2) -- (0, -2) -- (0,0) -- (1,0) -- (1, -1) -- (2, -1) -- (2, -2) -- (3, -2) -- (3, -3) ;
\draw[dotted] (0,-1) -- (1,-1) -- (1, -2) -- (2, -2) -- (2, -3) -- (3, -3) -- (3, -4);
\draw[dotted] (4,-5) -- (5, -5) -- (5, -6);
\draw[fill] (3, -3.5) circle [radius = .5 pt];
\draw[fill] (3.5, -4) circle [radius = .5 pt];
\draw[fill] (4, -4.5) circle [radius = .5 pt];
\draw (4, -4) --(5, -4) -- (5, -5) -- (6, -5) -- (6, -6) -- (4, -6) -- (4, -5);
\end{tikzpicture}
        }
        \caption{$\strata^{hyp}(2g-2)$}
        \label{fig:subfig1}
    \end{subfigure} \qquad
        \begin{subfigure}[b]{0.3\textwidth}
        \centering
        \resizebox{\linewidth}{!}{
\begin{tikzpicture}[scale = .75]
\draw (2, -4) -- (2, -3) -- (1, -3) -- (1, -2) -- (0, -2) -- (0,0) -- (1,0) -- (1, -1) -- (2, -1) -- (2, -2) -- (3, -2) -- (3, -3) ;
\draw[dotted] (0,-1) -- (1,-1) -- (1, -2) -- (2, -2) -- (2, -3) -- (3, -3) -- (3, -4);
\draw[dotted] (4,-5) -- (5, -5) -- (5, -6);
\draw[fill] (3, -3.5) circle [radius = .5 pt];
\draw[fill] (3.5, -4) circle [radius = .5 pt];
\draw[fill] (4, -4.5) circle [radius = .5 pt];
\draw (4, -4) --(5, -4) -- (5, -6) -- (4, -6) -- (4, -5);
\end{tikzpicture}
        }
        \caption{$\strata^{hyp}(g-1,g-1)$}
        \label{fig:subfig2}
    \end{subfigure}
\caption{Hyperelliptic Translation Surfaces} 
\label{fig:hyp-model}
\end{figure}

\subsection*{Generic Surfaces in Even Components of Strata}\label{Evens} \text{}

To find a surface in $\strata^{even}(2k_1, \hdots, 2k_n)$ start with the surface in Figure~\ref{fig4}, set $g = 1 + \sum_i k_i$, and then collapse every saddle connected labelled $a_i$ except those in $S:=\{a_0, a_{k_1}, a_{k_1 + k_2}, \hdots, a_{\sum_{i=1}^{n-1} k_i} \}$. This surface is in the even component since collapsing $S - \{a_0\}$ is a path in the stratum whose endpoint is a surface in the even minimal component (specifically the surface in Figure~\ref{fig:KZ-minB}). By the genericity criterion (Proposition~\ref{P:GC}) whenever the vertical cylinders have rationally unrelated moduli this surface is generic.

\begin{figure}[H]
\centering

\begin{tikzpicture}[scale = .75]
\draw (1,0) -- (0,0) -- (0,1) -- (1,1);
\draw (1,1) -- (1,2) -- (2,2) -- (2,1) -- (3,1) -- (3,2) -- (4,2) -- (4,1) -- (5,1) -- (5,2) -- (6,2) -- (6,1);
\draw[fill] (7.25,1.5) circle [radius = .5pt];
\draw[fill] (7.5,1.5) circle [radius = .5pt];
\draw[fill] (7.75,1.5) circle [radius = .5pt];
\draw (9,1) -- (9,2) -- (10,2) -- (10,1);
\draw (10,1) -- (11,1) -- (11,0) -- (10,0);
\node at (.5, 1.4) {$a_0$}; \node at (.5, -.4) {$a_0$};
\node at (1.25, 2.4) {$1$}; \node at (1.25, -.4) {$4$}; \draw[fill] (1.5,2) circle [radius = 2pt];
\node at (1.75, 2.4) {$2$}; \node at (1.75, -.4) {$3$}; \draw[fill] (1.5,0) circle [radius = 2pt];
\node at (2.5, 1.4) {$a_1$}; \node at (2.5, -.4) {$a_1$};
\node at (3.25, 2.4) {$3$}; \node at (3.25, -.4) {$2$}; \draw[fill] (3.5,2) circle [radius = 2pt];
\node at (3.75, 2.4) {$4$}; \node at (3.75, -.4) {$1$}; \draw[fill] (3.5,0) circle [radius = 2pt];
\node at (4.5, 1.4) {$a_2$}; \node at (4.5, -.4) {$a_2$};
\node at (5.5, 2.4) {$5$}; \node at (5.5, -.4) {$5$};
\node at (6.5, 1.4) {$a_3$}; \node at (6.5, -.4) {$a_3$};
\node at (8.4, 1.4) {$a_{g-2}$}; \node at (8.4, -.4) {$a_{g-2}$};
\node at (9.5, 2.4) {$g+1$}; \node at (9.6, -.4) {$g+1$};
\draw[ dotted] (1,1) -- (1,0);
\draw[ dotted] (2,1) -- (2,0);
\draw[ dotted] (3,1) -- (3,0);
\draw[ dotted] (4,1) -- (4,0);
\draw[ dotted] (5,1) -- (5,0);
\draw[dotted] (5,1) -- (6,1);
\draw[ dotted] (6,1) -- (6,0);
\draw[dotted] (8,1) -- (8,0);
\draw[dotted] (1,1) -- (4,1);
\draw (8,1) -- (9,1);
\draw[ dotted] (9,1) -- (9,0);
\draw[ thick] (2,1)--(3,1);
\draw[ thick] (2,0)--(3,0);
\draw[ thick] (4,1)--(5,1);
\draw[ thick] (4,0)--(5,0);
\draw (6,1) -- (7,1);
\draw[ dotted] (7,1) -- (7,0);
\draw[dotted] (9,1) -- (10,1) -- (10,0);
\draw (10,0) -- (1,0);
\end{tikzpicture}
\caption{ $\strata^{even}(2, \hdots, 2)$ }
\label{fig4}
\end{figure}

\begin{lemma}\label{LL5}
Let $(X, \omega)$ be the surface just constructed in an even component $\strata$ of a stratum of Abelian differentials. Let $C$ be the unique horizontal cylinder that intersects every vertical cylinder. If $g > 2$ and $\strata$ is not hyperelliptic, then there are two equivalence classes of vertical cylinder $\cD_1$ and $\cD_2$, so that $C \cap \cD_i$ is connected for $i = 1, 2$. Notice that $C - \left( \cD_1 \cup \cD_2 \right)$ consists of two disjoint rectangles. For generic choices of the lengths of the horizontal saddle connections, these two rectangles have different horizontal lengths. 
%Let $C$ be the horizontal cylinder that intersects every vertical cylinder. Let $\strata$ be the stratum containing the surface $(X,\omega)$ in Figure~\ref{fig4} and suppose it is not hyperelliptic. If $g > 2$ then there are two equivalence classes of vertical cylinders - $\cD_1$ and $\cD_2$ - whose intersection with $C$ is connected and so that  for generic choices of lengths of horizontal saddle connections, the horizontal length of the two components of $C - \cD_1 \cup \cD_2$ are different.
\end{lemma}
\begin{proof}
If $g > 4$ then we may set $\cD_1$ and $\cD_2$ equal to the vertical cylinder that passes through the horizontal saddle connection labelled $5$ and $6$ in Figure~\ref{fig4} respectively. Suppose now that $g \in \{3, 4\}$. 

Suppose first that the horizontal saddle connection labelled $a_1$ is uncollapsed. When $g = 4$, set $\cD_1$ and $\cD_2$ to be the vertical cylinders passing through the horizontal saddle connections labelled $a_1$ and $5$ respectively. When $g  = 3$ we take them to be the vertical cylinders passing through $a_0$ and $a_1$. 

Suppose now that $a_1$ is collapsed. Since $\strata$ is not hyperelliptic, $g = 4$. Set $\cD_1$ to be the equivalence class that contains the two cylinders that intersect the horizontal saddle connections labelled $\{1, 2, 3, 4\}$. Set $\cD_2$ to be the equivalence class of vertical cylinders intersecting the horizontal saddle connection labelled $5$.
\end{proof}

\subsection*{Generic Surfaces in Remaining Components}\label{Generals}

To find generic surfaces in all other connected components of the remaining strata we glue together copies of the surfaces in Figure~\ref{F:generic} along the horizontal cylinders that intersects all vertical cylinders. By the genericity criterion  (Proposition~\ref{P:GC}) whenever the vertical cylinders have rationally unrelated moduli this surface is generic.

\begin{figure}[H]
        \begin{subfigure}[b]{0.4\textwidth}
        \centering
        \resizebox{\linewidth}{!}{
\begin{tikzpicture}[scale = .75]
\draw (0,0) -- (0,1) -- (1,1) -- (1,2) -- (1.95,2) -- (1.95, 1) -- (2.05,1) -- (2.05,2) --  (2.95,2) -- (2.95,1) --  (3.05,1) -- (3.05,2) -- (4,2) -- (3.95, 2) -- (3.95, 1);
\draw[fill] (5,1.5) circle [radius = .5pt];
\draw[fill] (4.75,1.5) circle [radius = .5pt];
\draw[fill] (5.25,1.5) circle [radius = .5pt];
\draw (5.98,1) -- (5.98,2) ;
\draw (6.02,1) -- (6.02,2);
\draw (6,1) -- (6,2) -- (7,2) -- (7,1) -- (8,1) -- (8,0) -- (0,0);
\node at (1.5, 2.4) {$1$}; \node at (1.5, -.4) {$1$};
\node at (2.5, 2.4) {$2$}; \node at (2.5, -.4) {$2$};
\node at (3.5, 2.4) {$3$}; \node at (3.5, -.4) {$3$};
\node at (6.5, 2.4) {$r$}; \node at (6.5, -.4) {$r$};
\draw[dotted] (1,1) -- (1,0);
\draw[dotted] (2,1) -- (2,0);
\draw[dotted] (3,1) -- (3,0);
\draw[dotted] (4,1) -- (4,0);
\draw[dotted] (6,1) -- (6,0);
\draw[dotted] (7,1) -- (7,0);
\draw[dotted] (1,1) -- (4,1);
\draw[dotted] (6,1) -- (7,1);
\end{tikzpicture}
        }
        \caption{$\mathcal{H}(\hdots, 2r, \hdots)$}
        \label{F:generic1}
    \end{subfigure}
        \begin{subfigure}[b]{0.65\textwidth}
        \centering
        \resizebox{\linewidth}{!}{
\begin{tikzpicture}[scale = .75]
\draw (0,0) -- (0,1) -- (1,1) -- (1,2) -- (1.95,2) -- (1.95, 1) -- (2.05,1) -- (2.05,2) --  (2.95,2) -- (2.95,1) --  (3.05,1) -- (3.05,2) -- (4,2) -- (3.95, 2) -- (3.95, 1);
\draw[fill] (4.5,1.5) circle [radius = .5pt];
\draw[fill] (4.25,1.5) circle [radius = .5pt];
\draw[fill] (4.75,1.5) circle [radius = .5pt];
\draw (4.98,1) -- (4.98,2) ;
\draw (5.02,1) -- (5.02,2);
\draw (5, 2) -- (6, 2) -- (6, 1) -- (9, 1) -- (9, 2) -- (10,2) -- (10,1);
\draw (9.98,1) -- (9.98,2) ;
\draw (10.02,1) -- (10.02,2);
\draw[fill] (11,1.5) circle [radius = .5pt];
\draw[fill] (10.75,1.5) circle [radius = .5pt];
\draw[fill] (11.25,1.5) circle [radius = .5pt];
\draw (11.98,1) -- (11.98,2) ;
\draw (12.02,1) -- (12.02,2);
\draw (12, 2) -- (12.95,2) -- (12.95,1) -- (13.05,1) -- (13.05,2) -- (14,2) -- (14,1) -- (15, 1) -- (15, 0) -- (0,0);

\draw[dotted] (1,1) -- (1,0);
\draw[dotted] (2,1) -- (2,0);
\draw[dotted] (3,1) -- (3,0);
\draw[dotted] (5,1) -- (5,0);
\draw[dotted] (6,1) -- (6,0);
\draw[dotted] (7,1) -- (7,0);
\draw[dotted] (8, 1) -- (8, 0);
\draw[dotted] (9,1) -- (9,0);
\draw[dotted] (10,1) -- (10,0);
\draw[dotted] (12,1) -- (12,0);
\draw[dotted] (13,1) -- (13,0);
\draw[dotted] (14,1) -- (14,0);

\draw[dotted] (1,1) -- (3,1);
\draw[dotted] (6,1) -- (7,1);
\draw[dotted] (5,1) -- (6,1);
\draw[dotted] (9,1) -- (10,1);
\draw[dotted] (12,1) -- (14,1);

\node at (1.5, 2.4) {$a_1$}; \node at (1.5, -.4) {$a_1$};
\node at (2.5, 2.4) {$a_2$}; \node at (2.5, -.4) {$a_2$};
\node at (5.5, 2.4) {$a_n$}; \node at (5.5, -.4) {$a_n$};

\node at (6.5, 1.4) {$1$}; \node at (8.5, 1.4) {$2$};
\node at (6.5, -.4) {$2$}; \node at (8.5, -.4) {$1$};
\node at (7.5, 1.4) {$\alpha$}; \node at (7.5, -.4) {$\alpha$};

\node at (13.5, 2.4) {$b_1$}; \node at (13.5, -.4) {$b_1$};
\node at (12.5, 2.4) {$b_2$}; \node at (12.5, -.4) {$b_2$};
\node at (9.5, 2.4) {$b_m$}; \node at (9.5, -.4) {$b_m$};

\end{tikzpicture}
        }
        \caption{$\mathcal{H}(\hdots, 2n+1, 2m+1, \hdots)$}
        \label{F:generic2}
    \end{subfigure}
\caption{ Surfaces in $\mathcal{H}^{odd}$, $\mathcal{H}^{nonhyp}$, and connected strata}
\label{F:generic}
\end{figure}
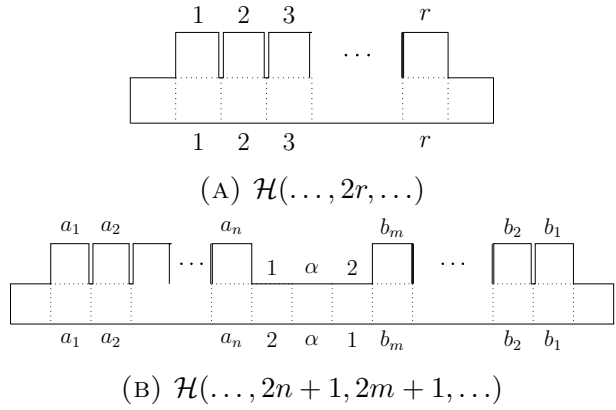

\begin{lemma}\label{LL51}
Lemma~\ref{LL5} holds for the surfaces just constructed in genus greater than two.
\end{lemma}
\begin{proof}
Suppose first that the surface has two zeros of even order. Then the surface contains two copies of the surface in Figure~\ref{F:generic1}. It is sufficient to take the two vertical cylinders that pass through the horizontal saddle connection labelled $1$. Similarly, if there are four zeros of odd order we have two copies of the surface in Figure~\ref{F:generic2} and we take the two vertical cylinders that pass through the horizontal saddle connection labelled $\alpha$.

If there is one zero of even order and two zeros of odd order, then we have a surface like the one in Figure~\ref{F:generic1} and one like the  one in Figure~\ref{F:generic2} and we take the vertical cylinder passing through the saddle connection labelled $1$ on the first and the one passing through the horizontal saddle connection labeled $\alpha$ on the second. 

If there is only one zero, then it is of even order and we take the two vertical cylinders that pass through $1$ and $2$. If there are only two zeros and both are of odd order then we take the vertical cylinder passing through $\alpha$ and the vertical cylinder passing through $a_1$ or $b_1$ (whichever exists). 
\end{proof}

\begin{defn}
Each of the nonhyperelliptic surfaces constructed in this section contain a horizontal cylinder that intersects every vertical cylinder. This cylinder will be called the central horizontal cylinder.
\end{defn}

%%%%%%%%%%%%%%%%%%%%%%%%%%
%
% SECTION - TECHNICAL MARKED POINTS LEMMA
%
%%%%%%%%%%%%%%%%%%%%%%%%%%

\section{Marked Points in Cylinders}~\label{S:Cylinder}

In this section we will prove results about marked points and cylinders that form the technical core of the paper. We will make the following standing assumption:

\begin{ass}
$\M$ is an affine invariant submanifold and $(X, \omega)$ is a translation surface whose $\GL(2, \R)$ orbit closure is $\M$
\end{ass}

\begin{lemma}\label{L:ec}
Let $P$ be a collection of distinct points on $(X, \omega)$ and suppose that $\M'$ is the orbit closure of $(X, \omega; P)$. If $\cC$ is an $\M$-equivalence class of cylinders on $(X, \omega)$, then $\cC'$ is an $\M'$-equivalence class where $\cC'$ contains the cylinders in $\cC$ divided into subcylinders by the points in $P$. 
\end{lemma}
\begin{proof}
First we will show that any two cylinders in $\cC'$ are $\M'$ equivalent. Let $C_i$ be two cylinders in $\cC'$ and let $\gamma_i$ be their core curves for $i = 1,2$. By assumption, there is a neighborhood $U$ of $(X, \omega)$ in $\M$ on which $\gamma_1$ and $\gamma_2$ are collinear. Let $U'$ be a preimage of the $U$ in $\M'$ on which $C_1$ and $C_2$ persist as cylinders. In this neighborhood, $\gamma_1$ and $\gamma_2$ must remain collinear and hence $\M'$-equivalent.

It remains to show that if $C_1$ and $C_2$ are $\M'$-equivalent cylinders with core curves $\gamma_1$ and $\gamma_2$, then $\gamma_1$ and $\gamma_2$ must be collinear on a neighborhood of $(X, \omega)$ in $\M$. By Lemma~\ref{L:neighbor} there is a neighborhood of $(X, \omega; P)$ in $\M'$ that projects to an open neighborhood of $(X, \omega)$ in $\M$. Let $U'$ be such a neighborhood of $(X, \omega; P)$ in $\M'$ where $\gamma_1$ and $\gamma_2$ are collinear and let $U$ be its image under the forgetful map. Since $\gamma_1$ and $\gamma_2$ are collinear on $U'$ they are collinear on $U$ and hence are homotopic to core curves of $\M$-equivalent cylinders.
\end{proof}

\begin{defn}
Given a cylinder $C$ in a translation surface we say that the height of the cylinder is the distance between the two boundaries in the flat metric. Suppose that $p$ is a marked point contained in a cylinder $C$. Let $h_C$ be the height of the cylinder and let $h_p$ be the distance from the point to one of the two boundary curves of $C$. We say that $p$ lies at rational height in $C$ if the ratio $\frac{h_p}{h_C}$ is rational. 
\end{defn}

\begin{lemma}[Rational Height Lemma]\label{L:RHL}
Let $\cC$ be an equivalence class of cylinders so that any two have a rational ratio of moduli. If a periodic point belongs to the interior of a cylinder in $\cC$ then it lies at rational height.
\end{lemma}
\begin{proof}
Let $p$ be a periodic point contained in the interior of a cylinder in $\cC$. Let $\M'$ be the orbit closure of $(X, \omega; p)$. Let $\cC'$ be the collection of subcylinders on $(X, \omega; p)$ into which $\cC$ is divided. By Lemma~\ref{L:ec}, $\cC'$ is an $\M'$ equivalence class. Let $\sigma_{\cC'}$ be the standard shear on $\cC'$. Since the cylinders in $\cC$ have a rational ratio of moduli, the flow along $\sigma_{\cC}$ is periodic. Suppose to a contradiction that $p$ does not have rational height. In this case, the flow along $\sigma_{\cC'}$ is not periodic and so the orbit closure of $(X, \omega; p)$ contains $(X, \omega; q)$ where $q$ is any point in $C$ along the core curve of $C$ that intersects $p$.

Let $\gamma_1$ and $\gamma_2$ be the two core curves of the cylinders into which $p$ divides $C$. Since $p$ may be moved along the core curve of $C$ while fixing all cylinders in $(X, \omega)$, it follows that the tangent space of $\M'$ at $(X, \omega; p)$ contains the deformation $\gamma_1^* - \gamma_2^*$. Let $U'$ be a neighborhood as in Lemma~\ref{L:neighbor} of $(X, \omega; p)$ in $\M'$ on which the cylinders in $\cC$ persist. This neighborhood projects to a neighborhood $U$ of $(X, \omega)$ in $\M$. Since the tangent space contains the deformation $\gamma_1^* - \gamma_2^*$ the fiber of the projection from $\M'$ to $\M$ that forgets marked points has real dimension at least one. Therefore, the dimension of $\M'$ is strictly larger than the dimension of $\M$, which contradicts the assumption that $p$ is a periodic point.  
\end{proof}

%The following will be a situation that we frequently find ourselves in when studying marked points on generic translation surfaces. Suppose that $(X, \omega)$ is a translation surface with orbit closure $\M$ and suppose that $q$ is a periodic point that is contained in the interior of a horizontal cylinder $C$. Suppose moreover that $C$ has height $h$ and is $\M$-free, meaning that it is the only cylinder in its $\M$-equivalence class of cylinders. We will suppose that $C$ is intersected by two vertical cylinders $D_1$ and $D_2$. Let $\gamma_i$ be the core curve of $D_i$ and suppose that the tangent space to $\M$ at $(X, \omega)$ contains the tangent vectors $\gamma_i^*$. 
%
%The situation is shown in: Figure~\ref{fig:shearing}:

\begin{lemma}\label{calc1}
Let $(X, \omega)$ be a generic translation surface in an affine invariant submanifold $\M$. Let $C$ be a horizontal cylinder, and let $\cD_1, \cD_2$ be two vertical distinct $\M$-equivalence classes of cylinders such that 
\begin{enumerate}
\item The intersection of $\cD_i$ with the interior of $C$ is connected and nonempty for $i = 1, 2$.
\item Any cylinder equivalent to $C$ has a modulus that is an integer multiple of the modulus of $C$. 
\end{enumerate}
If $p$ is an $\M$-periodic point in the interior of $C$, then up to relabelling $\cD_1$ and $\cD_2$, the point $p$ is at the center of the rectangle given by the intersection of $\cD_1$ and $C$. Furthermore, removing $\cD_1$ and $\cD_2$ divides $C$ into two rectangles of equal size. 
\begin{figure}[H]
            \begin{tikzpicture}[scale = .75]
                \draw[dashed] (0,3) -- (0,2);
                \draw (0,2) -- (0,0) -- (7,0);
                \draw[dashed] (1,3)--(1,0);
                \draw (1,2) -- (4,2);
                \draw[dashed] (4,3) -- (4,0);
                \draw[dashed] (5,3) -- (5,0);
                \draw (5,2) -- (7,2);
                \draw[black, fill] (.5, 1) circle[radius = 1.6pt];
                \node at (.5, .75) {$p$};
                \draw (-1,0) -- (-1,1) -- (-1.25, 1) -- (-.75, 1) -- (-1, 1) -- (-1,0) -- (-1.25, 0) -- (-.75, 0); \node at (-1.25, .5) {$h$};
                \draw (-2,0) -- (-2,2) -- (-2.25, 2) -- (-1.75, 2) -- (-2, 2) -- (-2,0) -- (-2.25, 0) -- (-1.75, 0); \node at (-2.25, 1) {$y$};
               %\draw (0,1) -- (7,1); \node at (4, 1.25) {$\gamma_C$};
		\node at (.5, 3) {$\cD_1$}; \node at (4.5, 3) {$\cD_2$}; \node at (7.5, 1) {$C$};
		\draw (0, -1) -- (.5, -1) -- (.5, -.75) -- (.5, -1.25) -- (.5, -1) -- (0, -1) -- (0, -1.25) -- (0, -.75); \node at (.25, -1.25) {$q \ell_1$};
		\draw (0, -2) -- (1, -2) -- (1, -1.75) -- (1, -2.25) -- (1, -2) -- (0, -2) -- (0, -1.75) -- (0, -2.25); \node at (.5, -2.25) {$ \ell_1$};
		\draw (1, -1) -- (4, -1) -- (4, -.75) -- (4, -1.25) -- (4, -1) -- (1, -1) -- (1, -1.25) -- (1, -.75); \node at (2.5, -1.25) {$a$};
		\draw (4, -2) -- (5, -2) -- (5, -1.75) -- (5, -2.25) -- (5, -2) -- (4, -2) -- (4, -1.75) -- (4, -2.25); \node at (4.5, -2.25) {$ \ell_2$};
		\draw (5, -1) -- (8, -1) -- (8, -.75) -- (8, -1.25) -- (8, -1) -- (5, -1) -- (5, -1.25) -- (5, -.75); \node at (6.5, -1.25) {$b$};
            \end{tikzpicture}
             \caption{The lemma shows that $a = b$ and that, after scaling so $C$ has unit height, $q = h = \frac{1}{2}$}
            \label{fig:shearing}
\end{figure}
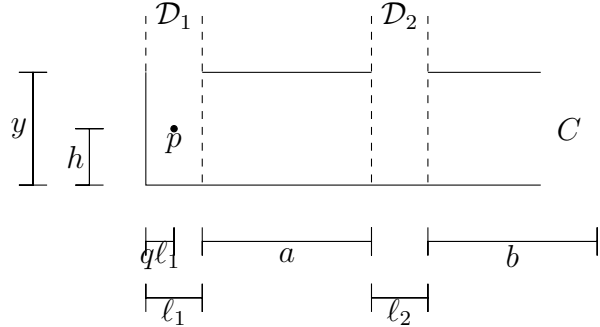
\end{lemma}
%
%In Figure~\ref{fig:shearing} we have let $y$ be the height of cylinder $C$ and $h$ the distance from the bottom boundary of cylinder $C$ to the marked point $q$. Let $a$ be the distance from $D_1$ to $D_2$ traveling left to right along the core curve of $C$ and let $b$ the distance from $D_1$ to $D_2$ traveling left to right. Let $\ell_i$ be the height of cylinder $D_i$. If, as shown in the picture, the marked point is contained in cylinder $D_1$, let $p\ell_1$ be the distance from the lefthand boundary of $D_1$ to the marked point $q$.  We will show the following
%
%\begin{lemma}\label{calc1}
%In the situation summarized in Figure~\ref{fig:shearing} the following hold:
%\begin{enumerate}
%\item The marked point lies a height a half in $C$, i.e. $\frac{h}{y} = \frac{1}{2}$
%\item The marked point is contained in $D_1$ (perhaps after permuting the labels $D_1$ and $D_2$).
%\item The lefthand and righthand distance between $D_1$ and $D_2$ coincide, i.e. $a = b$.
%\item The marked point lies at height a half in $D_1$, i.e. $p = \frac{1}{2}$.
%\end{enumerate}
%\end{lemma}
\begin{proof}
Let $\M'$ be the orbit closure of $(X, \omega; p)$. Let $\cC$ be the collection of horizontal cylinders equivalent to $C$ on $(X, \omega)$.

For simplicity we begin by applying a element of $\GL_2(\R)$ so that $C$ has unit height. By the rational height lemma (Lemma~\ref{L:RHL}) the marked point $p$ lies at rational height $h$ in $C$. Since we have normalized the height of $C$ to be one this means that $h$ is rational, in particular suppose that $h = \frac{n}{m}$ where $n$ and $m$ are coprime positive integers. 

\noindent \textbf{Part 1: We may assume that $p$ belongs to $\cD_1$}

Suppose not. By Lemma~\ref{L:ec}, since $p$ is not contained in $\cD_1$, the standard shear $\sigma_{\cD_1}$ is tangent to $\M'$. Traveling in the $\sigma_{\cD_1}$ direction in $\M'$ from $(X, \omega; p)$ widens $\cD_1$ while fixing the part of the translation surface (and marked point) in the complement $\cD_1$. Travel in this direction until the intersection of $\cD_1$ and $C$ accounts for at least $\frac{m-1}{m}$ proportion of the area of $C$. Let $(Y, \eta)$ be the new translation surface. 

 %Let $\sigma_i$ and $a_i$ be the standard shear and dilation respectively on $\cD_i$ for $i = 1, 2$. By Wright~\cite[Lemma 4.11]{Wcyl} (Theorem~\ref{T:W2}), these cohomology classes are tangent vectors to $\M$. By Lemma~\ref{L:ec}, since $p$ is not contained in $\cD_1$, $\sigma_1$ and $a_1$ are also tangent to $\M'$.

%For $i \in \{1, 2\}$, let $\sigma_i = \sum_{c \in \cD_i} h_c \gamma_c^*$ where given a cylinder $c$, its height is $h_c$ and $\gamma_c^*$ is the cohomology class that records oriented intersections with the core curve $\gamma_c$ of $c$ (oriented to travel downward). By the cylinder deformation theorem (Wright~\cite[Theorem 1.1]{Wcyl}), $\sigma_1$ and $\sigma_2$ are tangent vectors to $\M$. By Lemma~\ref{L:ec}, it is also a tangent vector to $\M'$. 

\noindent \textbf{Part 1a: $(Y, \eta)$ may be taken to be generic in $\M$}

We formed $(Y, \eta)$ by traveling in the $\sigma_{\cD_1}$ direction from $(X, \omega)$ in $\M$. Let $\ell$ be the segment joining $(X, \omega)$ to $(Y, \eta)$ in $\M$. Each proper affine invariant submanifold (of which there are only countably many by Eskin-Mirzakhani-Mohammadi~\cite{EMM}) contained in $\M$ intersects $\ell$ in a closed set. If a neighborhood $U$ of $(Y, \eta)$ had the property that every point in $U \cap \ell$ was contained in some proper affine invariant submanifold, then by the Baire category theorem there would be a proper affine invariant submanifold $\N$ so that $\N \cap \ell$ had interior in $U \cap \ell$. Since affine invariant submanifolds are linear this would imply that all of $\ell$ was contained in $\N$, which contradicts the fact that $\ell$ contains a translation surface $(X, \omega)$ that is generic in $\M$. Therefore, we may assume that the point $(Y, \eta)$ was chosen to be generic in $\M$.

\noindent \textbf{Part 1b: The hypotheses of the lemma continue to hold on $(Y, \eta)$ and the fiber of $\M'$ over $(Y, \eta)$ contains $(Y, \eta; p)$ where $p$ belongs to $\cD_1$}

Traveling in the $\sigma_{\cD_1}$ direction keeps $p$ fixed in the complement of $\cD_1 \cup \cD_2$ and keeps the heights of the cylinders in $\cC$ constant. By Wright~\cite[Theorem 1.9]{Wcyl}, the ratio of lengths of core curves of cylinders in $\cC$ are constant, and so the condition on moduli of cylinders in $\cC$ continues to hold. 

Letting $\cC'$ be the collection of cylinders on $(Y, \eta; p)$ that project to cylinders in $\cC$ we have that $\cC'$ is an $\M'$-equivalence class by Lemma~\ref{L:ec}. Travel from $(Y, \eta; p)$ in the $\sigma_{\cC'}$ direction until one complete Dehn twist has been performed in $C$. The resulting unmarked translation surface is $(Y, \eta)$ since all cylinders in $\cC$ have moduli that are integer multiples of the modulus of $C$. Since $\cD_1 \cap C$ is connected and accounts for at least $\frac{m-1}{m}$ of the area of $C$, it follows that the marked point $p$ now belongs to $\cD_1$ and we have shown that $(Y, \eta; p)$ belongs to $\M'$. 

%and has the property that $(Y, \eta)$ is generic in $\M$ and satisfies the hypotheses of the Lemma. 

%By the cylinder deformation theorem~\cite[Theorem 1.1]{Wcyl}, $\sigma_{\cC} := \sum_{c \in \cC} h_c \gamma_c^*$ is tangent to $\M$. Let $\cC'$ be the collection of subcylinders that $p$ divides $\cC$ into. By Lemma~\ref{L:ec}, $\sigma_{\cC'} := \sum_{c \in \cC} h_c \gamma_c$ is tangent to $\M'$. Flow along $\sigma_{\cC'}$ to perform one complete Dehn twist in $C$. Since all cylinders in $\cC$ have a modulus that is an integer multiple of the modulus of $C$, we have that the resulting translation surface is again $(Y, \eta)$. Since $\cD_1 \cap C$ is connected and accounts for at least $\frac{m-1}{m}$ of the area of $C$, it follows that the marked point $p$ now belongs to $\cD_1$ and we have shown that $(Y, \eta; p)$ belongs to $\M'$ and has the property that $(Y, \eta)$ is generic in $\M$ and satisfies the hypotheses of the Lemma. 

\noindent \textbf{Part 1c: If the conclusion of the lemma holds on $(Y, \eta)$, it does so on $(X, \omega)$ as well}

To pass from $(X, \omega)$ to $(Y, \eta)$, we first traveled along $\sigma_{\cD_1}$ and then traveled along $\sigma_{\cC'}$ to perform one Dehn twist in $C$. If the conclusion of the lemma holds on $(Y, \eta)$, then traveling in the opposite direction along $\sigma_{\cC'}$ to perform the opposite Dehn twist in $C$ moves $p$ to the midpoint of the rectangle $\cD_2 \cap C$. Let $\cD_1'$ be the collection of subcylinders into which the cylinders in $\cD_1$ are divided by $p$. Traveling along $\sigma_{\cD_1'}$ back to $(X, \omega)$ keeps the complement of $\cD_1$ fixed and so the conclusion of the lemma held on $(X, \omega; p)$ as desired. Therefore, we may suppose without loss of generality (by replacing $(X, \omega)$ with $(Y, \eta)$) that $p$ belongs to $\cD_1$.

%To pass from $(X, \omega)$ to $(Y, \eta)$, we first traveled along $\sigma_1$ and then traveled along $\sigma_{\cC'}$ to perform one Dehn twist in $C$. Therefore, traveling in the opposite direction along $\sigma_1$ returns to $(X, \omega)$, while keeping the marked point fixed at the same height in $\cD_1$ and $C$. 

%We will show now that if we can establish the result for $(Y, \eta)$, then it must also hold for the original surface $(X, \omega)$. To pass from $(X, \omega)$ to $(Y, \eta)$, we first flowed along $\sigma_1$ and then flowed along $\sigma_{\cC'}$ to perform one Dehn twist in $C$. Therefore, flowing in the opposite direction along $\sigma_{\cC'}$ to perform the opposite Dehn twist in $C$ we see that if the conclusion of this Lemma holds, then $p$ becomes the point at the center of the intersection of $\cD_2$ and $C$. Now flowing along $\sigma_1$ back to $(X, \omega)$ keeps the complement of $\cD_1$ in $C$ fixed and so we see that the conclusion of the Lemma holds on $(X, \omega)$ as well. 
%

\noindent \textbf{Part 2: Determining the position of $p$}

Suppose now that $p$ belongs to $\cD_1$. As before, $\sigma_{\cD_2}$ is tangent to $\M'$. Travel from $(X, \omega; p)$ in the $\sigma_{\cD_2}$ direction in $\M'$ until the intersection of $\cD_2$ and $C$ accounts for at least $\frac{m-1}{m}$ proportion of the area of $C$. Without loss of generality, we may replace $(X, \omega; p)$ with the resulting marked translation surface. 

Let $\ell_i$ be the horizontal length of the rectangle $C \cap \cD_i$ for $i = 1,2$. The complement of $\cD_1 \cup \cD_2$ in $C$ is two disjoint rectangles. Let $a$ (resp. $b$) be horizontal length of the rectangle to the right (resp. left) of $\cD_1 \cap C$, see Figure~\ref{fig:shearing}. Let $\ell$ be the length of the core curve of $C$. Let $q \in [0, 1]$ be chosen so that $p$ is a distance of $q \ell_1$ from the left boundary of $\cD_1 \cap C$. Let $\cD_1'$ be the collection of subcylinders that $p$ divides $\cD_1$ into on $(X, \omega; p)$. 

Travel in the $\sigma_{\cD_1'}$ direction from $(X, \omega; p)$ so that the length of the core curve of $C$ increases by $s$ and then travel in the $\sigma_{\cC'}$ direction to perform exactly one Dehn twist in $C$. The distance of the marked point from the lefthand boundary of $\cD_2 \cap C$ is the following,
\[ h \left( \ell + s \right) - (1-q) (\ell_1 +s) - a \]
Traveling back along the $\sigma_{\cD_1'}$ direction returns to the unmarked surface $(X, \omega)$ while leaving the position of the marked point fixed in the complement of $\cD_1$. Since $p$ is a periodic point, the fiber of the forgetful map from $\M'$ to $\M$ over $(X, \omega)$ is finite.

Therefore, $h\left( \ell + s \right) -  (1-q) (\ell_1 +s) - a$ is constant as a function of $s$. In other words, $h = (1-q)$. If we sheared $C$ in the other direction we would have by symmetry that $h = q$ and so $q = h = \frac12$. By symmetry, after shearing the marked point into $\cD_2$ the distance from the lefthand boundary of $\cD_2 \cap C$ is $\ell_2/2$, i.e.
\[ \frac{1}{2} \left( \ell - \ell_1 \right) - a = \frac{\ell_2}{2} \]
Since $\ell = \ell_1 + a + \ell_2 + b$ we see that $a = b$ as desired.
\end{proof}

%%%%%%%%%%%%%%%%%%%%
%
% SECTION - Reduction
%
%%%%%%%%%%%%%%%%%%%%

\section{Proof of Theorem~\ref{T1}}\label{S:T1}

Throughout this section, we make the following assumption.

\begin{ass}\label{A:MT}
Let $\M$ be an affine invariant submanifold in $\strata(0^n)$ where $\strata$ is an unmarked stratum of Abelian differentials. Suppose that $\M$ contains marked points on a translation surface $(X, \omega)$ that is generic in $\strata$ and is one of the surfaces constructed in Section~\ref{S:KZ-Surfaces}. Finally, suppose after passing to a finite cover that the marked points are labelled as $\{p_1, \hdots, p_n\}$. Let $\pi_k: \strata^{ord}(0^n) \ra \strata^{ord}(0^{n-1})$ be the map that forgets the $k$th marked point for $k \in \{1, \hdots, n\}$. 
\end{ass}

\begin{thm}\label{TS1}
Periodic points exist on $(X, \omega)$ if and only if $\strata$ is hyperelliptic, in which case they are Weierstrass points. 
\end{thm}

\begin{thm}\label{TS2}
If $n \geq 2$ and $\left( \pi_k \right)_* \M = \strata^{ord}(0^{n-1})$ for all $k \in \{1, \hdots, n\}$ then $\strata$ is hyperelliptic, $n=2$, and the fiber of $\M$ over $(X, \omega)$ contains all pairs of distinct points exchanged by the hyperelliptic involution.
\end{thm}
\noindent If $\strata$ is hyperelliptic, then given a pair of integers $\{i,j\}$ integers in $\{1, \hdots, n\}$, let $\strata_{ij}$ denote the subset of $\strata(0^n)$ where $p_i$ and $p_j$ are exchanged by the hyperelliptic involution. If $i = j$, this will means that $p_i$ is a fixed point of the hyperelliptic involution. We will prove the following strengthening of Theorem~\ref{T1}. 

\begin{thm}\label{TS3}
The stratum $\strata$ is hyperelliptic and there is a subset $S$ of pairs of integers in $\{1, \hdots, n\}$ so that $\ds{ \M = \bigcap_{\{i, j\} \in S} \strata_{ij} }$. 
\end{thm}
\begin{proof}[Proof of Theorem~\ref{TS3} given Theorem~\ref{TS1} and Theorem~\ref{TS2}:]
Proceed by induction on $n$. The $n=1$ case is Theorem~\ref{TS1}. Suppose now that $n>1$. 

Suppose first that for some $k \in \{1, \hdots, n\}$, $(\pi_k)_* \M$ has dimension $\dim \M -1$. Suppose without loss of generality after relabelling that $k = 1$. By the induction hypothesis, $\strata$ is hyperelliptic and $\ds{ (\pi_1)_* \M = \bigcap_{\{i, j\} \in S} \strata_{ij}(0^{n-1}) }$ where $S$ is some subset of pairs of integers in $\{2, \hdots, n\}$. It follows that $\M$ is contained in $\ds{ \bigcap_{\{i, j\} \in S} \strata_{ij}(0^n) }$ and therefore coincides with it since both manifolds are connected, closed, and of the same dimension.

Suppose now that for all $k \in \{1, \hdots, n\}$, $\left( \pi_k \right)_* \M$ has the same dimension as $\M$ and that it does not coincide with $\strata(0^{n-1})$. By the induction hypothesis, for each $k$, there is a subset $S_k$ of pairs of integers in $\{1, \hdots, n\} - \{k \}$, so that $\ds{ \left( \pi_k \right)_* \M  = \bigcap_{\{i, j\} \in S_k} \strata_{ij} }$. The number of elements of $S_k$ is exactly the codimension of $\left( \pi_k \right)_* \M$ in $\strata(0^{n-1})$. Suppose without loss of generality that $\{1, \ell_1\}$ is contained in $S_{\ell_2}$. Then this pair cannot be contained in $S_1$ and so $S_1 \cup S_{\ell_2}$ contains at least as many elements as the codimension of $\M$ in $\strata(0^n)$. Hence, $\ds{ \M = \bigcap_{\{i, j\} \in S_1 \cup S_{\ell_2}} \strata_{ij}(0^{n}) }$ since both manifolds are connected, closed, and of the same dimension.

The only case that remains is when $(\pi_k)_* \M = \strata^{ord}(0^{n-1})$ for all $k \in \{1, \hdots, n\}$, which follows from Theorem~\ref{TS2}. 
\end{proof}

\section{Proof of Theorem~\ref{TS2}}

Assumption~\ref{A:MT} will remain in effect for this section. Assume too that $\M$ is a proper affine invariant submanifold in $\strata^{ord}(0^n)$ where $n \geq 2$ and suppose that $\left( \pi_k \right)_* \M = \strata^{ord}(0^{n-1})$ for all $k \in \{1, \hdots, n\}$. If $\strata$ is non-hyperelliptic then let $C$ be the central horizontal cylinder and if $\strata$ is hyperelliptic let $C$ be any horizontal cylinder that intersects two vertical cylinders. By Lemma~\ref{LL5} and~\ref{LL51}, there are two equivalence classes of vertical cylinders $\cD_1$ and $\cD_2$ whose intersection with $C$ is connected and nonempty.

\begin{lemma}\label{L2}
There are two marked points. If one marked point lies in a cylinder in $\cD_1$ or $\cD_2$, the other one lies in that cylinder as well. 
\end{lemma}
\begin{proof}
 Let $P = \{p_1, \hdots, p_{n-1} \}$ be a collection of $n-1$ points where $p_1$ lies in $\cD_1$ and divides the cylinders in $\cD_1$ into subcylinders whose moduli admit no rational homogeneous linear relation. Suppose too that $\{p_2, \hdots, p_{n-1} \}$ belong to $\cD_2$ and divide it into vertical sub-cylinders whose moduli also admit no rational homogeneous linear relation. The genericity criterion (Proposition~\ref{P:GC}) implies that $(X, \omega; P)$ is generic in $\strata(0^{n-1})$. By Lemma~\ref{L:fiber}, there is a point $p_n$ so that $(X, \omega; P \cup \{ p_n \} )$ belongs to $\M$. 

Let $\cD$ be an equivalence class in $\{\cD_1, \cD_2\}$ that does not contain $p_n$. Since $\cD$ is its own equivalence class in $\strata$ and since it is divided into sub-cylinders whose moduli admit no rational homogeneous linear relation, it follows that each subcylinder in $\cD$ may be sheared while fixing the rest of the surface. Phrased differently, in the fiber of $\M$ over $(X, \omega)$ any marked point in $\cD$ may be moved freely while fixing all other points in $P$. However, if $p_k$ is a point (for $k \in \{1, \hdots, n-1\}$) that belongs to $\cD$, then the fiber of $\pi_k : \M \ra \strata^{ord}(0^{n-1})$ is one-dimensional and, by assumption $\left( \pi_k \right)_* \M = \strata^{ord}(0^{n-1})$. This implies that $\M = \strata^{ord}(0^n)$, which is a contradiction. Therefore, there are no points belonging to $\cD$ and so $n = 2$. 

For the final statement, suppose to a contradiction that $\{p_1, p_2\}$ is a fiber of $\M$ over $(X, \omega)$ under the map that forgets marked points and suppose too that $p_1$ belongs to $\cD$ for $\cD \in \{\cD_1, \cD_2 \}$, but that $p_2$ does not. Since the map $\pi_2$ that forgets the second point is open by Lemma~\ref{L:neighbor}, there is a nearby surface $(X, \omega; p_1', p_2')$ in $\M$ where $p_1'$ divides the cylinders in $\cD$ into subcylinders whose moduli admit no rational homogeneous linear relation and so that $p_2'$ remains outside of $\cD$. This contradicts the previous paragraph. 
\end{proof}

Since $\pi_2: \M \ra \strata(0)$ is a finite holomorphic map, we see that given a point $(X, \omega; p_1, p_2) \in \M$ we move $p_1$ to a new point $p_1'$ (at least locally) and there will be a unique nearby point $(X, \omega; p_1', p_2') \in \M$. Since the equations that define the affine invariant submanifold $\M$ have real coefficients, if $p_1$ moves horizontally (resp. vertically), so does $p_2$.

Now suppose without loss of generality that $p_1$ lies at irrational height in $\cD_1$ and at irrational height in $C$. Let $p_2$ be a point so that $(X, \omega; p_1, p_2) \in \M$. By Lemma~\ref{L2}, $p_2$ must also lie in the interior $\cD_1 \cap C$, which is a rectangle. Moving $p_1$ to the left we see that $p_1$ reaches the left boundary of $\cD_1 \cap C$ at the same moment that $p_2$ reaches the vertical boundary. Now reversing direction and moving $p_1$ to the right we see again that $p_1$ and $p_2$ arrive at the vertical boundary of $\cD_1$ at the same moment. This implies either that one point lies above the other and both move at the same speed in the same direction (horizontally) or that both points at some point were on opposite boundaries of $\cD_1$ and move in opposite directions (horizontally) at the same speed. The first case cannot occur, since if it does we may simply shear the central horizontal cylinder and find two marked points that do not lie above each other and that still move in the same direction at the same speed. The same argument applied to the vertical direction shows that when one point moves at unit speed in the $v$ direction, the other point moves at unit speed in the $-v$ direction. Coupled with the fact that the points arrive at the $\cD_1 \cap C$ boundary at the same times we have that when $\strata$ is hyperelliptic the two points are exchanged by the hyperelliptic involution.

We will now show that $\strata$ must be hyperelliptic. The present situation is pictured in Figure~\ref{fig:shearing2}. If $\strata$ is not hyperelliptic then we may ensure that $a < b$ (by Lemmas~\ref{LL5} and~\ref{LL51}). If $p_2$ moves to the right at unit speed, then $p_1$ moves to the left at unit speed and hence $p_2$ arrives in the interior of $\cD_2$ before $p_1$, contradicting Lemma~\ref{L2}.
\begin{figure}[H]
            \begin{tikzpicture}[scale = .75]
                \draw[dashed] (0,3) -- (0,2);
                \draw (0,2) -- (0,0) -- (7,0);
                \draw[dashed] (1,3)--(1,0);
                \draw (1,2) -- (4,2);
                \draw[dashed] (4,3) -- (4,0);
                \draw[dashed] (5,3) -- (5,0);
                \draw (5,2) -- (7,2);
                \draw[black, fill] (0, 1.5) circle[radius = 1.6pt]; \node at (.4, 1.5) {$p_1$};
                \draw[black, fill] (1, .5) circle[radius = 1.6pt]; \node at (.6, .5) {$p_2$};
                \draw (-2,0) -- (-2,2) -- (-2.25, 2) -- (-1.75, 2) -- (-2, 2) -- (-2,0) -- (-2.25, 0) -- (-1.75, 0); \node at (-2.25, 1) {$1$};
		\node at (.5, 3) {$\cD_1$}; \node at (4.5, 3) {$\cD_2$}; \node at (7.5, 1) {$C$};
		\draw (0, -2) -- (1, -2) -- (1, -1.75) -- (1, -2.25) -- (1, -2) -- (0, -2) -- (0, -1.75) -- (0, -2.25); \node at (.5, -2.25) {$ \ell_1$};
		\draw (1, -1) -- (4, -1) -- (4, -.75) -- (4, -1.25) -- (4, -1) -- (1, -1) -- (1, -1.25) -- (1, -.75); \node at (2.5, -1.25) {$a$};
		\draw (4, -2) -- (5, -2) -- (5, -1.75) -- (5, -2.25) -- (5, -2) -- (4, -2) -- (4, -1.75) -- (4, -2.25); \node at (4.5, -2.25) {$ \ell_2$};
		\draw (5, -1) -- (8, -1) -- (8, -.75) -- (8, -1.25) -- (8, -1) -- (5, -1) -- (5, -1.25) -- (5, -.75); \node at (6.5, -1.25) {$b$};
            \end{tikzpicture}
             \caption{Two marked points in the central horizontal cylinder}
            \label{fig:shearing2}
\end{figure}
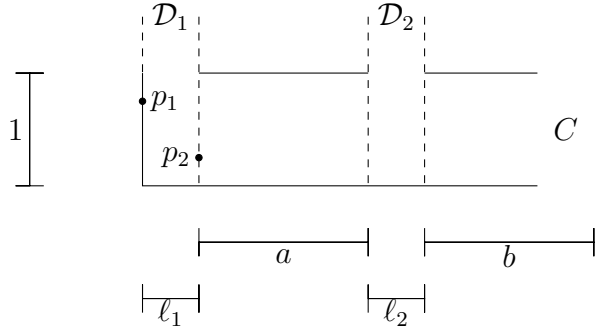

%%%%%%%%%%%%%%%%%%%%
%
% SECTION - MINIMAL STRATA
%
%%%%%%%%%%%%%%%%%%%%

\section{Proof of Theorem~\ref{TS1}}

Throughout this section we will make the following assumption.

\begin{ass}\label{A-periodic}
Let $\strata$ be a stratum of Abelian differentials and let $(X, \omega)$ be a generic translation surface in $\strata$ constructed in Section~\ref{S:KZ-Surfaces}. Let $\M$ be an affine invariant submanifold properly contained in $\strata(0)$. By Lemma~\ref{L:fiber}, the fiber in $\M$ over $(X, \omega)$ is nonempty and any point $p$ in the fiber is a periodic point.
\end{ass}

\begin{defn}
A cylinder $C$ is called $\strata$-free if $\{ C \}$ is an $\strata$-equivalence class. This is equivalent to the condition that there is no other parallel cylinder $C$ with a core curve that is homologous to the core curve of $C$.
\end{defn}

\begin{prop}\label{P:hyp}
The periodic points on $\hyp(2g-2)$ and $\hyp(g-1,g-1)$ are exactly the Weierstrass points.
\end{prop}
\begin{proof}
Let $\strata$ be either $\hyp(2g-2)$ or $\hyp(g-1,g-1)$. The surface $(X, \omega)$ is pictured again in Figure~\ref{fig:hyp}

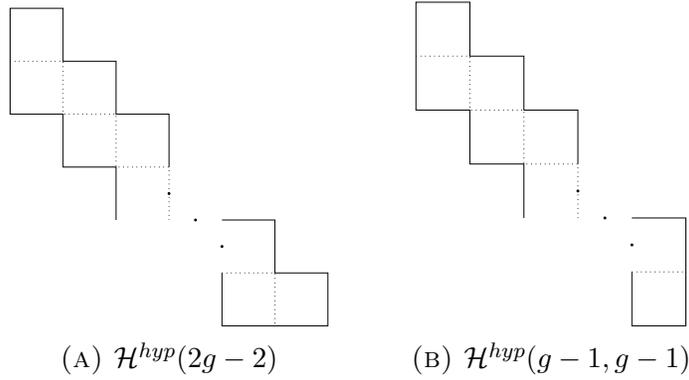
\begin{figure}[H]
    \begin{subfigure}[b]{.35\textwidth}
        \centering
        \resizebox{\linewidth}{!}{
\begin{tikzpicture}
\draw (2, -4) -- (2, -3) -- (1, -3) -- (1, -2) -- (0, -2) -- (0,0) -- (1,0) -- (1, -1) -- (2, -1) -- (2, -2) -- (3, -2) -- (3, -3) ;
\draw[dotted] (0,-1) -- (1,-1) -- (1, -2) -- (2, -2) -- (2, -3) -- (3, -3) -- (3, -4);
\draw[dotted] (4,-5) -- (5, -5) -- (5, -6);
\draw[fill] (3, -3.5) circle [radius = .5 pt];
\draw[fill] (3.5, -4) circle [radius = .5 pt];
\draw[fill] (4, -4.5) circle [radius = .5 pt];
\draw (4, -4) --(5, -4) -- (5, -5) -- (6, -5) -- (6, -6) -- (4, -6) -- (4, -5);
\end{tikzpicture}
        }
        \caption{$\strata^{hyp}(2g-2)$}
        \label{fig:subfig1}
    \end{subfigure} \qquad
        \begin{subfigure}[b]{.3\textwidth}
        \centering
        \resizebox{\linewidth}{!}{
\begin{tikzpicture}
\draw (2, -4) -- (2, -3) -- (1, -3) -- (1, -2) -- (0, -2) -- (0,0) -- (1,0) -- (1, -1) -- (2, -1) -- (2, -2) -- (3, -2) -- (3, -3) ;
\draw[dotted] (0,-1) -- (1,-1) -- (1, -2) -- (2, -2) -- (2, -3) -- (3, -3) -- (3, -4);
\draw[dotted] (4,-5) -- (5, -5) -- (5, -6);
\draw[fill] (3, -3.5) circle [radius = .5 pt];
\draw[fill] (3.5, -4) circle [radius = .5 pt];
\draw[fill] (4, -4.5) circle [radius = .5 pt];
\draw (4, -4) --(5, -4) -- (5, -6) -- (4, -6) -- (4, -5);
\end{tikzpicture}
        }
        \caption{$\strata^{hyp}(g-1,g-1)$}
        \label{fig:subfig2}
    \end{subfigure}
\caption{Hyperelliptic Translation Surfaces} 
\label{fig:hyp}
\end{figure}

Each horizontal and vertical cylinder is $\strata$-free. By Lemma~\ref{calc1} if $p$ is a periodic point in $(X, \omega)$ that lies in the interior of a horizontal cylinder that intersects two vertical cylinders, it is automatically a Weierstrass point. The same holds if it lies in the interior of a vertical cylinder that intersects two horizontal ones. We may therefore assume (up to exchanging each instance of the word ``horizontal" for the word ``vertical" and vice versa) that $p$ lies in the interior of a horizontal cylinder that intersects only one vertical cylinder and on the boundary of a vertical cylinder. By Lemma~\ref{L:ec}, we may shear this cylinder and remain in $\M$. Shearing the horizontal cylinder so as to perform one complete Dehn twist moves the periodic point into the interior of a vertical cylinder that intersects two horizontal cylinders and so we have that $p$ is a Weierstrass point by Lemma~\ref{calc1}. 
\end{proof}

\begin{ass}\label{A-periodic2}
Assume now that $\strata$ is nonhyperelliptic and let $C$ be the central horizontal cylinder in $(X, \omega)$. 
\end{ass}

\begin{prop}\label{P:almost}
A periodic point on $(X, \omega)$ must lie on the boundary of $C$ and in the interior of a vertical cylinder that is not $\strata$-free.
\end{prop}
\begin{proof}
Let $p$ be a periodic point in $(X, \omega)$. By Lemma~\ref{LL5}, Lemma~\ref{LL51}, and Lemma~\ref{calc1} the periodic point cannot lie in the interior of $C$. We will proceed now by cases based on the containment of $p$ in vertical cylinders.

\noindent \textbf{Case 1: $p$ is contained in a vertical cylinder $V$ that is $\strata$-free, is contained in $C$, and only intersects the core curve of $C$ once}

In this case, Lemma~\ref{L:ec} implies that we may shear $V$ so as to perform one complete Dehn twist and remain in $\M$. This moves $p$ to a periodic point in the interior of $C$, which is a contradiction.

\noindent \textbf{Case 2: $p$ is contained in a vertical cylinder $V$ that is $\strata$-free and is contained in $C$}

By the previous case, $V$ must intersect the core curve of $C$ at least twice. By construction of the surfaces in Section~\ref{S:KZ-Surfaces} the situation must be as depicted in Figure~\ref{F:free2}. The marked point is then contained in a $\strata$-free cylinder (drawn in dashed lines). Using Lemma~\ref{L:ec} to shear this cylinder to perform one complete Dehn twist we see that $p$ may be moved to a periodic point in the interior of $C$, which is a contradiction. 

\begin{figure}[H]
\centering
\begin{tikzpicture}[scale=.75]
\draw (5.5, 1) -- (9.5,1);
\draw (5.5,0) -- (9.5,0);
\draw[dotted] (6,1) -- (6,0);
\draw[dotted] (7,1) -- (7,0);
\draw[dotted] (8, 1) -- (8, 0);
\draw[dotted] (9,1) -- (9,0);

\node at (6.5, 1.4) {$1$}; \node at (8.5, 1.4) {$2$};
\node at (6.5, -.4) {$2$}; \node at (8.5, -.4) {$1$};
%\node at (7.5, -.6) {$W$};
\node at (5.5, .5) {$C$};

\draw[fill] (7,1) circle [radius = 4pt]; \draw[fill] (7,0) circle [radius = 4pt];
\draw[fill] (9,1) circle [radius = 4pt]; \draw[fill] (9,0) circle [radius = 4pt];
\draw (6,1) circle [radius = 4pt]; \draw (8,1) circle [radius = 4pt];
\draw (6,0) circle [radius = 4pt]; \draw (8,0) circle [radius = 4pt];

\draw (6.5, 1) circle [radius = 2pt]; \node at (6.5, .6) {$p$};

\draw[dashed] (6,1) -- (8,0); \draw[dashed] (7,1) -- (9,0);
\end{tikzpicture}
\caption{The translation surface in Case 2}
\label{F:free2}
\end{figure}
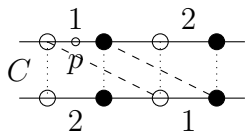

\noindent \textbf{Case 3: $p$ is contained in a $\strata$-free vertical cylinder $V$ that is not contained in $C$}

By construction of the surfaces in Section~\ref{S:KZ-Surfaces}, $V$ only intersects two horizontal cylinders - $C$ and $H$ - and core curves intersect exactly once, see Figure~\ref{F:free3}. Applying Lemma~\ref{calc1}, with $\cD_1 = \{C\}$ and $\cD_2 = \{H\}$ we see that if $p$ lies in the interior of $V$, then there must be a periodic point in $C$, which is a contradiction. If $p$ does not lie in the interior of $V$, then by Lemma~\ref{L:ec} we may shear $H$ to perform one complete Dehn twist while fixing the remainder of the translation surface and remaining in $\M$. This shear moves $p$ to the interior of $V$ and so we are done. 

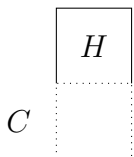
\begin{figure}[H]
\centering
	\begin{tikzpicture}
		\draw[dotted] (0,0) -- (0,1) -- (1,1) -- (1,0); 
		\draw (0,0) -- (1,0);
		\draw (0,1) -- (0,2) -- (1,2) -- (1,1);
		\node at (.5, 1.5) {$H$}; \node at (-.5, .5) {$C$};
	\end{tikzpicture}
\caption{The vertical cylinder $V$}
\label{F:free3}
\end{figure}

\noindent \textbf{Case 4: $p$ is contained in a vertical cylinder that is not contained in $C$ and that is not $\strata$-free}

By construction of the surfaces in Section~\ref{S:KZ-Surfaces}, $\strata$ is an even component of a stratum of Abelian differentials. Moreover, either $p$ is contained on the boundary of a vertical cylinder, as in Figure~\ref{fig:exceptions} or is contained in one of the cylinders passing through the saddle connections labelled $\{1, \hdots, 4\}$ in Figure~\ref{fig:exceptions}.

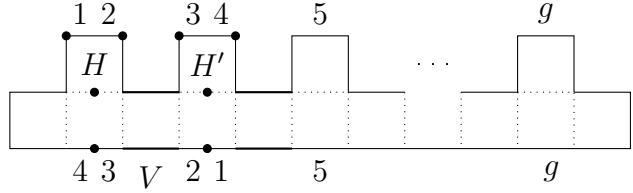
\begin{figure}[H]
\centering

\begin{tikzpicture}[scale = .75]
\draw (1,0) -- (0,0) -- (0,1) -- (1,1);
\draw (1,1) -- (1,2) -- (2,2) -- (2,1) -- (3,1) -- (3,2) -- (4,2) -- (4,1) -- (5,1) -- (5,2) -- (6,2) -- (6,1);
\draw[fill] (7.25,1.5) circle [radius = .25pt];
\draw[fill] (7.5,1.5) circle [radius = .25pt];
\draw[fill] (7.75,1.5) circle [radius = .25pt];
\draw (9,1) -- (9,2) -- (10,2) -- (10,1);
\draw (10,1) -- (11,1) -- (11,0) -- (10,0);
]\node at (1.25, 2.4) {$1$}; \node at (1.25, -.4) {$4$}; \draw[fill] (1,2) circle [radius = 2pt];\draw[fill] (2,2) circle [radius = 2pt];
\node at (1.75, 2.4) {$2$}; \node at (1.75, -.4) {$3$}; \draw[fill] (1.5,0) circle [radius = 2pt];\draw[fill] (1.5,1) circle [radius = 2pt];
]\node at (3.25, 2.4) {$3$}; \node at (3.25, -.4) {$2$}; \draw[fill] (3,2) circle [radius = 2pt];\draw[fill] (4,2) circle [radius = 2pt];
\node at (3.75, 2.4) {$4$}; \node at (3.75, -.4) {$1$}; \draw[fill] (3.5,0) circle [radius = 2pt]; \draw[fill] (3.5,1) circle [radius = 2pt];
]\node at (5.5, 2.4) {$5$}; \node at (5.5, -.4) {$5$};
\node at (9.5, 2.4) {$g$}; \node at (9.6, -.4) {$g$};
\node at (1.5, 1.5) {$H$}; \node at (3.5, 1.5) {$H'$}; \node at (2.5, -.5) {$V$};
\draw[ dotted] (1,1) -- (1,0);
\draw[ dotted] (2,1) -- (2,0);
\draw[ dotted] (3,1) -- (3,0);
\draw[ dotted] (4,1) -- (4,0);
\draw[ dotted] (5,1) -- (5,0);
\draw[dotted] (5,1) -- (6,1);
\draw[ dotted] (6,1) -- (6,0);
\draw[dotted] (8,1) -- (8,0);
\draw[dotted] (1,1) -- (4,1);
\draw (8,1) -- (9,1);
\draw[ dotted] (9,1) -- (9,0);
\draw[ thick] (2,1)--(3,1);
\draw[ thick] (2,0)--(3,0);
\draw[ thick] (4,1)--(5,1);
\draw[ thick] (4,0)--(5,0);
\draw (6,1) -- (7,1);
\draw[ dotted] (7,1) -- (7,0);
\draw[dotted] (9,1) -- (10,1) -- (10,0);
\draw (10,0) -- (1,0);
\end{tikzpicture}
\caption{ Points on the boundary of $C$ and on the boundary of a vertical cylinder }
\label{fig:exceptions}
\end{figure}

Let $H$ and $H'$ be the indicated horizontal cylinder in Figure~\ref{fig:exceptions}, which are $\strata$-free.  Suppose without loss of generality that $p$ is contained in the horizontal cylinder $H$ or its boundary. By Lemma~\ref{L:ec} we may shear the cylinders $H$ and $H'$ to arrive at the surface in Figure~\ref{fig:exceptions0}. Let $D$ be the diagonal cylinder with dashed boundary that passes through the horizontal saddle connection labelled $1$. By Lemma~\ref{L:ec} we may shear $H$ if necessary to perform one complete Dehn twist and move $p$ into the interior of $D$ while remaining in $\M$. By Lemma~\ref{calc1} - where $D$ is intersected by the equivalence classes $\{ H \}$ and $\{C\}$ - it follows that there is a periodic point contained in the interior of $C$. By Lemma~\ref{L:ec}, we may shear $H$ and $H'$ while fixing the remainder of the translation surface to return to $(X, \omega)$ with a periodic point $p$ in the interior of $C$, a contradiction. 

%By Lemma~\ref{L:ec} we may also increase the height of $C$ while fixing the rest of the translation surface. Since $D$ is $\strata$-free, we may shear it by Lemma~\ref{L:ec} to perform one complete Dehn twist. By making $C$ sufficiently large we may guarantee that this moves $p$ into the interior of $C$. Applying Lemma~\ref{L:ec} we may return $C$ to its original height (while keeping the marked point in its interior) and unshear the horizontal cylinders $H$ and $H'$. This shows that $C$ contains a periodic point on the original translation surface, which is a contradiction.

\begin{figure}[H]
\centering

\begin{tikzpicture}[scale = .75]
\draw (1,0) -- (0,0) -- (0,1) -- (1,1);
\draw (1,1) -- (-1,2) -- (0,2) -- (2,1) -- (3,1) -- (3,2) -- (4,2) -- (4,1) -- (5,1) -- (5,2) -- (6,2) -- (6,1);
\draw[dashed] (2,1)--(4, 0);
\draw[dashed] (1,1)--(3,0);
\draw[fill] (7.25,1.5) circle [radius = .25pt];
\draw[fill] (7.5,1.5) circle [radius = .25pt];
\draw[fill] (7.75,1.5) circle [radius = .25pt];
\draw (9,1) -- (9,2) -- (10,2) -- (10,1);
\draw (10,1) -- (11,1) -- (11,0) -- (10,0);
\node at (-.5, 2.5) {$1$}; \node at (3.5, -.5) {$1$}; \node at (1.5, -.5) {$2$}; \node at (3.5, 2.5) {$2$};
%\draw[fill] (-.5,2) circle [radius = 2pt];
%\draw[fill] (1.5,1) circle [radius = 2pt];
\node at (5.5, 2.4) {$5$}; \node at (5.5, -.4) {$5$};
\node at (9.5, 2.4) {$g$}; \node at (9.6, -.4) {$g$};
%\node at (.5, 1.5) {$H$}; \node at (4.5, 1.5) {$H'$}; \node at (2.5, -.5) {$V$};
\draw[ dotted] (1,1) -- (1,0);
\draw[ dotted] (2,1) -- (2,0);
\draw[ dotted] (3,1) -- (3,0);
\draw[ dotted] (4,1) -- (4,0);
\draw[ dotted] (5,1) -- (5,0);
\draw[dotted] (5,1) -- (6,1);
\draw[ dotted] (6,1) -- (6,0);
\draw[dotted] (8,1) -- (8,0);
\draw[dotted] (1,1) -- (4,1);
\draw (8,1) -- (9,1);
\draw[ dotted] (9,1) -- (9,0);
\draw[ thick] (2,1)--(3,1);
\draw[ thick] (2,0)--(3,0);
\draw[ thick] (4,1)--(5,1);
\draw[ thick] (4,0)--(5,0);
\draw (6,1) -- (7,1);
\draw[ dotted] (7,1) -- (7,0);
\draw[dotted] (9,1) -- (10,1) -- (10,0);
\draw (10,0) -- (1,0);
\end{tikzpicture}
\caption{Moving potentially periodic points into the interior of $C$ }
\label{fig:exceptions0}
\end{figure}
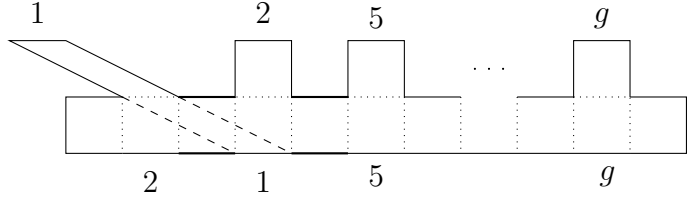

Therefore, $p$ lies on the boundary of $C$ and is contained in the interior of a vertical cylinder which is not $\strata$-free. 
\end{proof}

\begin{cor}\label{C:almost}
The only component of $\strata(2g-2)$ with periodic points is $\hyp(2g-2)$.
\end{cor}
\begin{proof}
Let $\strata$ be a minimal nonhyperelliptic stratum of Abelian differentials. If $\M$ is an affine invariant submanifold properly contained in $\strata(0)$ and that pushes forward to $\strata$ under the forgetful map, then its fiber over $(X, \omega)$ is nonempty by Lemma~\ref{L:fiber}. If $(X, \omega; p)$ is an element of the fiber then $p$ is contained in a vertical cylinder that is contained in $C$ and that is not $\strata$-free. By construction of the surfaces in Section~\ref{S:KZ-Surfaces} there are no such cylinders and so we are done. 
\end{proof}

\begin{proof}[Proof of Theorem~\ref{TS1}:]
Let $\strata$ be a component of a stratum of Abelian differentials with at least two zeros. Proceed by induction on $\dim_\C \strata$. The result has already been established for hyperelliptic components (Proposition~\ref{P:hyp}), which establishes the base case, and allows us to assume without loss of generality that $\strata$ is not a hyperelliptic component. Assume that $p$ is a periodic point on $(X, \omega)$. By Proposition~\ref{P:almost}, $p$ is contained on the boundary of $C$ and in a vertical cylinder $V$ that is contained in $C$ and that is not $\strata$-free.

\noindent \textbf{Case 1: $C$ contains an $\strata$-free vertical cylinder $W$}

By construction of the surfaces in Section~\ref{S:KZ-Surfaces}, there are only two types of $\strata$-free vertical cylinders contained in $C$, they are pictured in Figure~\ref{fig:free-cylinder}.

\begin{figure}[H]
\centering
    \begin{subfigure}{0.33\textwidth}
	\begin{tikzpicture}[scale = .75]

\draw (5.5, 1) -- (9.5,1);
\draw (5.5,0) -- (9.5,0);
\draw[dotted] (6,1) -- (6,0);
\draw[dotted] (7,1) -- (7,0);
\draw[dotted] (8, 1) -- (8, 0);
\draw[dotted] (9,1) -- (9,0);

\node at (6.5, 1.4) {$1$}; \node at (8.5, 1.4) {$2$};
\node at (6.5, -.4) {$2$}; \node at (8.5, -.4) {$1$};
\node at (7.5, -.6) {$W$};
\node at (5.5, .5) {$C$};

\draw[fill] (7,1) circle [radius = 4pt]; \draw[fill] (7,0) circle [radius = 4pt];
\draw[fill] (9,1) circle [radius = 4pt]; \draw[fill] (9,0) circle [radius = 4pt];
\draw (6,1) circle [radius = 4pt]; \draw (8,1) circle [radius = 4pt];
\draw (6,0) circle [radius = 4pt]; \draw (8,0) circle [radius = 4pt];
	\end{tikzpicture}
	\caption{A subsurface of $(X, \omega)$ in components with two zeros of odd order}
\end{subfigure}
\qquad 
    \begin{subfigure}{0.33\textwidth}
	\begin{tikzpicture}[scale = .75]
		\draw (4,0) -- (0,0) -- (0,1) -- (1,1) -- (1,2) -- (2,2) -- (2,1) -- (3,1) -- (3,2) -- (4,2) --(4,1);
		\draw[dotted] (1,0) -- (1,1) -- (2,1) -- (2,0);
		\draw[dotted] (4,0) -- (4,1) -- (3,1) -- (3,0);
		\node at (1.5, 2.4) {$1$}; \node at (1.5, -.4) {$2$};
		\node at (3.5, 2.4) {$2$}; \node at (3.5, -.4) {$1$};
		\node at (2.5, -.4) {$W$};
	\end{tikzpicture}
	\caption{A subsurface of $(X, \omega)$ in even components}
\end{subfigure}
\caption{The two types of $\strata$-free cylinders $W$ in $C$}
\label{fig:free-cylinder}
\end{figure}

By Lemma~\ref{L:ec}, if $(X, \omega)$ contains an $\strata$-free cylinder $W$ we may travel in $\M$ in the direction of the standard shear $\sigma_W$ to shrink the horizontal cross curve of $W$ until it vanishes. This degenerates $\M$ to the boundary of $\strata(0)$. Notice that in both cases this degeneration causes two distinct zeros to collide on the boundary, but, by considering Euler characteristic, no genus is lost. By Mirzakhani-Wright~\cite[Corollary 1.2]{MirWri}, the resulting translation surface has an orbit closure of strictly smaller dimension than $\M$. By the genericity criterion (Proposition~\ref{P:GC}), the boundary translation surface remains generic in the component of the stratum to which it belongs, which necessarily has complex dimension one less than $\M$. Therefore, $p$ remains a periodic point. 

By the induction hypothesis, $p$ is a Weierstrass point and the boundary translation surface belongs to a hyperelliptic component. In particular, $p$ must lie halfway across $V$ on the boundary of $C$, dividing $V$ into two subcylinders of equal modulus, call them $V_1$ and $V_2$. By construction, on $(X, \omega;p)$ the only rational linear homogeneous equation that holds on moduli of cylinders equivalent to $V$ is that the moduli of $V_1$ and $V_2$ are equal. By Wright~\cite{Wcyl} - in particular Theorems~\ref{T:W1} and \ref{T:W2} - and Lemma~\ref{L:ec}, $V$ may be sheared on $(X, \omega;p)$ while remaining in $\M$ and fixing the remainder of the translation surface. Shearing so as to perform one complete Dehn twist moves $p$ into the interior of $C$ on $(X, \omega)$, which contradicts Proposition~\ref{P:almost}.

Notice that as a corollary of this step, $\strata$ does not contain two zeros of odd order.

\noindent \textbf{Case 2: The surface belongs to an even component}

Let $H$ and $H'$ be the horizontal cylinders labelled in Figure~\ref{fig:exceptions}. By Lemma~\ref{L:ec}, shearing them while fixing the remainder of the surface remains in $\M$. Therefore, we shear them to find the surface in $\M$ depicted in Figure~\ref{fig:even-again}, which contains a vertical cylinder that contains $H$ and $H'$ and that passes through them exactly once.

\begin{figure}[H]
\begin{tikzpicture}[scale = .75]
\draw (3,0) -- (0,0) -- (0,1) -- (1,1) -- (1,2) -- (1.95,2) -- (1.95,1) -- (2.05, 1) -- (2.05, 2) -- (3, 2) -- (3,1);
\node at (1.5, 2.4) {$1$}; \node at (1.5, -.4) {$2$}; 
\node at (2.5, 2.4) {$2$}; \node at (2.5, -.4) {$1$}; 
\draw[ dotted] (1,1) -- (1,0);
\draw[ dotted] (2,1) -- (2,0);
\draw[ dotted] (3,1) -- (3,0);
\draw[ dotted] (3,1) -- (4,1);
\draw[ dotted] (3,0) -- (4,0);
\draw (4,1) -- (5,1);
\draw (4,0) -- (5,0);
\draw[ dotted] (4,1) -- (4,0);
\draw[ dotted] (5,1) -- (5,0);
\draw[ dotted] (5,1) -- (6,1);
\draw[ dotted] (5,0) -- (6,0);
\draw[fill] (4.5,1) circle [radius = 2pt]; \node at (4.5, 1.35) {$p$}; \node at (4.5, -.35) {$V$};
\node at (1.5, 1.5) {$H'$}; \node at (2.5, 1.5) {$H$};
\end{tikzpicture}
\caption{A translation surface in the even component}
\label{fig:even-again}
\end{figure}
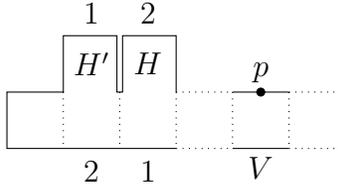

By Lemma~\ref{L:ec}, applying the standard dilation $a_{\{ H \}}$ to $H$ causes its vertical cross curve to vanish and passes to a surface on the boundary of $\M$. When the cross curve vanishes, the boundary translation surface (shown in Figure~\ref{fig:even-again-collapse}) has the zero of order $2k+2$ (where $k$ is a positive integer) on the boundary of $H$ split into two zeros - one of order $1$ and one of order $2k-1$. 

\begin{figure}[H]
\begin{tikzpicture}[scale = .75]
\draw (3,0) -- (0,0) -- (0,1) -- (1,1) -- (1,2) -- (1.95,2) -- (1.95,1) -- (3,1);
\node at (1.5, 2.4) {$1$}; \node at (1.5, -.4) {$2$}; 
\node at (2.5, 1.4) {$2$}; \node at (2.5, -.4) {$1$}; 
\draw[ dotted] (1,1) -- (1,0);
\draw[ dotted] (2,1) -- (2,0);
\draw[ dotted] (3,1) -- (3,0);
\draw[ dotted] (3,1) -- (4,1);
\draw[ dotted] (3,0) -- (4,0);
\draw (4,1) -- (5,1);
\draw (4,0) -- (5,0);
\draw[ dotted] (4,1) -- (4,0);
\draw[ dotted] (5,1) -- (5,0);
\draw[ dotted] (5,1) -- (6,1);
\draw[ dotted] (5,0) -- (6,0);
\draw[fill] (4.5,1) circle [radius = 2pt]; \node at (4.5, 1.35) {$p$}; \node at (4.5, -.35) {$V$};
\end{tikzpicture}
\caption{The boundary translation surface $(Y, \eta)$}
\label{fig:even-again-collapse}
\end{figure}
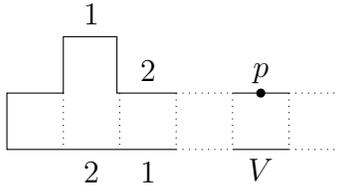

Let $\M'$ be the orbit closure of $(Y, \eta; p)$. By the genericity criterion (Proposition~\ref{P:GC}), $(Y, \eta)$ remains generic. By Mirzakhani-Wright~\cite[Corollary 1.2]{MirWri}, the dimension of $\M'$ must be strictly smaller than the dimension of $\M$. It follows that $p$ is a periodic point on $(Y, \eta)$. However, no such points exist on translation surfaces of genus greater than two with a zero of odd order. Therefore, $(X, \omega)$ cannot belong to an even component.

\noindent \textbf{Case 3: The surface does not belong to an even component}

The marked point is contained in one of the two configurations shown in Figure~\ref{fig:two-config}, where $V$ is adjacent on the right to a vertical cylinder $W$ that contains a horizontal cylinder $H$ (both shown in the figure). 

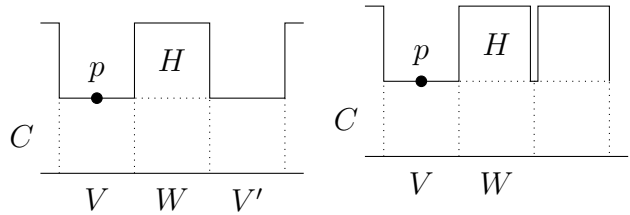
\begin{figure}[H]
    \begin{subfigure}{0.33\textwidth}
        \centering
\begin{tikzpicture}
		\draw (-.25, 0) -- (3.25,0);
		\draw (-.25, 2) -- (0,2) -- (0, 1) -- (1,1) -- (1,2) -- (2,2) -- (2,1) -- (3,1) -- (3,2) -- (3.25, 2);
		\draw[dotted] (0,0) -- (0,1);
		\draw[dotted] (1,0) -- (1,1) -- (2,1) -- (2,0);
		\draw[dotted] (3,1) -- (3,0);
		\node at (.5, -.35) {$V$}; \node at (.5, 1.35) {$p$};
		\draw[fill] (.5,1) circle [radius = 2pt];
		\node at (-.5, .5) {$C$}; \node at (1.5, -.35) {$W$};
		\node at (2.5, -.35) {$V'$}; \node at (1.5, 1.5) {$H$};
\end{tikzpicture}
\caption{The periodic point borders an order $2$ zero}
\end{subfigure}
\begin{subfigure}{0.33\textwidth}
\centering
\begin{tikzpicture}
		\draw (-.25, 0) -- (3.25,0);
		\draw (-.25, 2) -- (0,2) -- (0, 1) -- (1,1) -- (1,2) -- (1.95,2) -- (1.95,1) -- (2.05,1) -- (2.05,2) -- (3, 2) -- (3,1);
		\draw[dotted] (0,0) -- (0,1);
		\draw[dotted] (1,0) -- (1,1) -- (2,1) -- (2,0);
		\draw[dotted] (2,1) -- (3,1) -- (3,0);
		\node at (.5, -.35) {$V$}; \node at (.5, 1.35) {$p$};
		\draw[fill] (.5,1) circle [radius = 2pt];
		\node at (-.5, .5) {$C$}; \node at (1.5, -.35) {$W$}; \node at (1.5, 1.5) {$H$};
\end{tikzpicture}
\caption{The periodic point borders a zero of order $4$ or more}
\end{subfigure}
\caption{Two possible configurations} 
\label{fig:two-config}
\end{figure}

By Lemma~\ref{L:ec}, applying the standard dilation $a_{\{ H \}}$ to $H$ causes its vertical cross curve to vanish and passes to a surface on the boundary of $\M$. The underlying translation surface moves from $\strata(2k, \hdots)$ to $\strata(0, 2k-2, \hdots)$ (see Figure~\ref{fig:collapse-config}). 

\begin{figure}[H]
    \begin{subfigure}{0.33\textwidth}
        \centering
\begin{tikzpicture}
		\draw (-.25, 0) -- (3.25,0);
		\draw (-.25, 2) -- (0,2) -- (0, 1) -- (3,1) -- (3,2) -- (3.25, 2);
		\draw[dotted] (0,0) -- (0,1);
		\draw[dotted] (1,0) -- (1,1) -- (2,1) -- (2,0);
		\draw[dotted] (3,1) -- (3,0);
		\node at (.5, -.35) {$V$}; \node at (.5, 1.35) {$p$};
		\draw[fill] (.5,1) circle [radius = 2pt];
		\node at (-.5, .5) {$C$}; \node at (1.5, -.35) {$W$};
		\node at (2.5, -.35) {$V'$};
		\draw[fill]  (1,1) circle [radius = 2pt]; \draw[fill] (2,1) circle [radius = 2pt];
\end{tikzpicture}
\caption{The surfaces collapses to $\strata(0,0,\hdots)$}
\label{F:end1}
\end{subfigure}
\begin{subfigure}{0.33\textwidth}
\centering
\begin{tikzpicture}
		\draw (-.25, 0) -- (3.25,0);
		\draw (-.25, 2) -- (0,2) -- (0, 1) -- (2.05,1) -- (2.05,2) -- (3, 2) -- (3,1);
		\draw[dotted] (0,0) -- (0,1);
		\draw[dotted] (1,0) -- (1,1) -- (2,1) -- (2,0);
		\draw[dotted] (2,1) -- (3,1) -- (3,0);
		\node at (.5, -.35) {$V$}; \node at (.5, 1.35) {$p$};
		\draw[fill] (.5,1) circle [radius = 2pt];
		\node at (-.5, .5) {$C$}; \node at (1.5, -.35) {$W$};
		\draw[fill] (1,1) circle [radius = 2pt]; 
\end{tikzpicture}
\caption{The surfaces collapses to $\strata(0,2k-2,\hdots)$}
\label{F:end2}
\end{subfigure}
\caption{The result of collapsing $H$} 
\label{fig:collapse-config}
\end{figure}
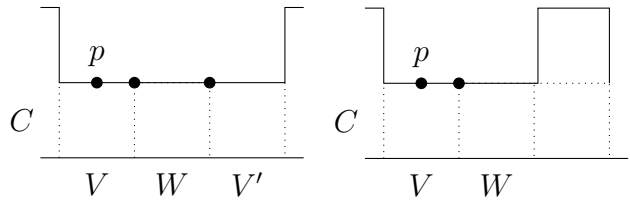

Let $(Y, \eta; \{ p \} \cup Q)$ be the boundary translation surface where $Q$ are the marked points that arise in the degeneration. Let $\M'$ be the orbit closure of $(Y, \eta; \{ p \} \cup Q)$. The genericity criterion (Proposition~\ref{P:GC}) implies that $(Y, \eta; Q)$ is generic in the stratum $\strata'$ that contains it. By Mirzakhani-Wright~\cite[Corollary 1.2]{MirWri}, $\M'$ is an affine invariant submanifold that is properly contained in $\strata'(0)$. The induction hypothesis implies that $(Y, \eta)$ belongs to a hyperelliptic component and that either $p$ is a Weierstrass point or is exchanged with a point in $Q$ under the hyperelliptic involution. 

Consider first the configuration in Figure~\ref{F:end2}. Since $V \cup W$ must be fixed by the hyperelliptic involution and since $W$ may be made arbitrarily long horizontally we see that $p$ is neither a Weierstrass point nor a point exchanged under the hyperelliptic involution with a point in $Q$.

Consider now the configuration in Figure~\ref{F:end1} and let $V_a$ and $V_b$ be the left and right sub-cylinders that $p$ splits $V$ into. We see that $p$ must be exchanged under the hyperelliptic involution with the rightmost point in $Q$ and so $V_a$ and $V'$ have identical moduli. Repeating the argument with the vertical cylinder $W'$ that $V$ borders on the left shows that $(X, \omega)$ must contain the subsurface shown in Figure~\ref{F:end} and satisfy the property that the modulus of $V'$ is the same as the modulus of $V_a$ and the modulus of $V''$ the same as the modulus of $V_b$.

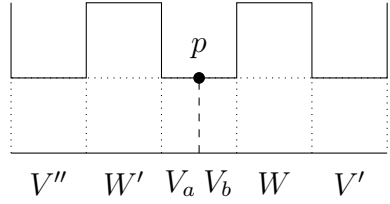
\begin{figure}[H]
        \centering
		\begin{tikzpicture}
			\draw (-1, 2) -- (-1, 1) -- (0,1) -- (0,2) -- (1,2) -- (1,1) -- ( 2,1) -- (2,2) -- (3,2) -- ( 3,1) -- (4,1) -- (4,2);
			\draw (-1,0) -- (4,0);
			\draw[dotted] (-1,0) -- (-1,1) -- (4,1) -- (4,0);
			\draw[dotted] (0,0) -- (0,1); \draw[dotted] (1,0) -- (1,1); 
			\draw[dotted] (2,0) -- (2,1); \draw[dotted] (3,0) -- (3,1); 
			\node at (-.5, -.4) {$V''$};  \node at (.5, -.4) {$W'$}; 
			\node at (2.5, -.4) {$W$};  \node at (3.5, -.4) {$V'$}; 
			\node at (1.5, 1.4) {$p$}; \draw[fill] (1.5,1) circle [radius = 2pt];
			\node at (1.25, -.4) {$V_a$}; \node at (1.75, -.4) {$V_b$};
			\draw[dashed] (1.5, 0) -- (1.5, 1);
		\end{tikzpicture}
		\caption{The surface $(X, \omega; p)$} 
		\label{F:end}
\end{figure} 

Letting $\mathrm{Mod}(D)$ denote the modulus of a cylinder $D$ we see that on this surface
\[ \mathrm{Mod}(V) = \mathrm{Mod}(V_a) + \mathrm{Mod}(V_b) = \mathrm{Mod}(V') + \mathrm{Mod}(V'') \]
which is a rational linear homogeneous relation on moduli satisfied by vertical cylinders on $(X, \omega)$, which contradicts the fact that $(X, \omega)$ is constructed so as to prevent the moduli of vertical cylinders from satisfying such a relation.  \end{proof}

\bibliography{mybib}{}
\bibliographystyle{amsalpha}
\end{document}